\documentclass[11pt,leqno]{article}

\usepackage{amsmath,amsthm}

\usepackage{amssymb,latexsym}

\usepackage{enumerate}

\newcommand{\AAA}{{\cal A}}
\newcommand{\BB}{{\cal B}}
\newcommand{\DD}{{\cal D}}
\newcommand{\EE}{{\cal E}}
\newcommand{\FD}{{\cal FD}}
\newcommand{\FF}{{\cal F}}

\newcommand{\MM}{{\cal M}}
\newcommand{\NN}{{\cal N}}

\newcommand{\RR}{{\cal R}}
\newcommand{\TT}{{\cal T}}

\newcommand{\BD}{{\mathbb D}}
\newcommand{\BN}{{\mathbb{N}}}
\newcommand{\BR}{{\mathbb{R}}}
\newcommand{\BX}{{\mathbb{X}}}
\newcommand{\dyw}{\mbox{\rm div}}
\newcommand{\fch}{\mathbf{1}}
\newcommand{\tr}{\mbox{\rm Tr}}

\newtheorem{theorem}{\bf Theorem}[section]
\newtheorem{proposition}[theorem]{\bf Proposition}
\newtheorem{lemma}[theorem]{\bf Lemma}
\newtheorem{corollary}[theorem]{\bf Corollary}

\theoremstyle{definition}
\newtheorem*{definition}{Definition}

\newtheorem{remark}[theorem]{Remark}

\numberwithin{equation}{section}

\begin{document}

\title{Semilinear elliptic equations with measure data and
quasi-regular Dirichlet forms}
\author{Tomasz  Klimsiak  and Andrzej Rozkosz
\mbox{}\\[2mm]
{\small  Faculty of Mathematics and Computer Science,
Nicolaus Copernicus University}\\
{\small Chopina 12/18, 87-100 Toru\'n, Poland} \\
{\small E-mail addresses: tomas@mat.umk.pl, rozkosz@mat.umk.pl }}
\date{}
\maketitle

\footnote{2010 \emph{Mathematics Subject Classification}: Primary
35J61, 35R06; Secondary 60H30.}

\footnote{\emph{Key words and phrases}: semilinear elliptic
equation, measure data, Dirichlet form, backward sto\-chastic
differential equation.}

\begin{abstract}
We are mainly concerned with equations of the form
$-Lu=f(x,u)+\mu$, where $L$ is an operator associated with a
quasi-regular possibly nonsymmetric Dirichlet form, $f$ satisfies
the monotonicity condition and mild integrability conditions, and
$\mu$ is a bounded smooth measure. We prove general results on
existence, uniqueness and regularity of probabilistic solutions,
which are expressed in terms of solutions to backward stochastic
differential equations. Applications include equations with
nonsymmetric divergence form operators, with gradient
perturbations of some pseudodifferential operators and equations
with Ornstein-Uhlenbeck type operators in Hilbert spaces. We also
briefly discuss the existence and uniqueness of probabilistic
solutions in the case where $L$ corresponds to a lower bounded
semi-Dirichlet form.
\end{abstract}

\section{Introduction}

Let $E$ be a metrizable Lusin space, $m$ be a positive
$\sigma$-finite  measure on  $\BB(E)$ and let $(\EE,D(\EE))$ be a
quasi-regular possibly nonsymmetric Dirichlet form on $L^2(E;m)$.
In the present paper we study existence, uniqueness and regularity
of solutions of semilinear equations of the form
\begin{equation}
\label{eq1.1} -Lu=f(x,u)+\mu.
\end{equation}
Here $f:E\times\BR\rightarrow\BR$ is a measurable function, $\mu$
is a smooth signed measure on $\BB(E)$ with respect to the
capacity determined by $\EE$, and $L$ is the operator associated
with the form $\EE$, i.e.
\begin{equation}
\label{eq1.2} (-Lu,v)=\EE(u,v),\quad u\in D(L),\,v\in D(\EE),
\end{equation}
where $D(L)=\{u\in D(\EE):v\mapsto\EE(u,v)$ is continuous with
respect to $(\cdot,\cdot)^{1/2}$ on $D(\EE)$\}. We assume that $f$
satisfies the  monotonicity condition and mild integrability
conditions (even weaker than the integrability conditions
considered earlier in \cite{BBGGPV}). As for $\mu$ we assume that
it belongs to the class
\begin{align}
\label{eq1.03} \mathcal{R}&=\{\mbox{$\mu:|\mu|$ is smooth and
$\hat{G}\phi\cdot\mu\in\mathcal{M}_{0,b}$ for}\\
&\qquad\qquad \mbox{some $\phi\in L^1(E;m)$ such that $\phi>0$\,\,
$m$-a.e.}\}, \nonumber
\end{align}
where $|\mu|$ denotes the variation of $\mu$, $\mathcal{M}_{0,b}$
is a space of all finite smooth signed measures and $\hat{G}$ is
the co-potential operator associated with $\EE$.
In the important case where $\EE$ is transient the class
$\mathcal{R}$ includes $\mathcal{M}_{0,b}$ but it may happen that
$\mathcal{R}$ also includes some Radon measures of infinite total
variation.

The paper continuous research begun in our paper \cite{KR:JFA} in
which equations of the form (\ref{eq1.1}) with $L$ associated with
symmetric regular Dirichlet form are studied. The main motivation
is to extend results of \cite{KR:JFA} to encompass equations with
non-symmetric operators and equations in infinite dimensions.

As in \cite{KR:JFA}, by a solution of (\ref{eq1.1}) we mean a
quasi-continuous function $u:E\rightarrow\BR$ satisfying for
quasi-every $x\in E$ the nonlinear Feynman-Kac formula
\begin{equation}
\label{eq1.3} u(x)=E_{x}\Big(\int_{0}^{\zeta}f(X_{t},u(X_{t}))\,dt
+\int_{0}^{\zeta}dA^{\mu}_{t}\Big),
\end{equation}
where $\BX=(X,P_x)$ is a Markov process with life-time $\zeta$
associated with the form $\EE$, $E_x$ is the expectation with
respect to $P_x$ and $A^{\mu}$ is the additive functional of $\BX$
corresponding to $\mu$ in the Revuz sense. We show that in the
case where $\EE$ is transient the solution may be defined in
purely analytic terms resembling Stampacchia's definition of
solutions by duality. Namely, a solution of (\ref{eq1.1}) can be
defined equivalently as a quasi-continuous function $u$ such that
$|\langle\nu,u\rangle|=|\int_Eu\,d\nu|<\infty$ for every $\nu$ in
the set $\hat S^{(0)}_{00}$ of smooth measures of $0$-order energy
integral such that $\|\hat U\nu\|_{\infty}<\infty$ and
\[
\langle\nu,u\rangle=(f(\cdot,u),\hat U\nu)
+\langle\mu,\widetilde{\hat U\nu}\rangle,\quad \nu\in \hat
S^{(0)}_{00},
\]
where $(\cdot,\cdot)$ is the usual scalar product in $L^2(E;m)$,
$\hat U\nu$  is the 0-order co-potential of $\nu$ and
$\widetilde{\hat U\nu}$ denotes its quasi-continuous version. We
work exclusively with the probabilistic definition (\ref{eq1.3})
because in our opinion it is simpler and more natural than the
definition by duality, and what is even more important, it allows
us to use directly powerful methods of the theory of Dirichlet
forms and Markov processes.

The paper is organized as follows. In Section \ref{sec2} we
provide basic definitions and prove some auxiliary but important
results on smooth measures and their associated additive
functionals.

In Section \ref{sec3} we prove the existence and uniqueness of
probabilistic solutions of (\ref{eq1.1}), and then in Section
\ref{sec4} we study additional regularity properties of the
solutions. Our main result says that under mild assumptions on
$f$, we have $f(\cdot,u)\in L^1(E;m)$, and for every $k>0$ the
truncation $T_ku:= (-k)\vee u\wedge k$ belongs to the extended
Dirichlet space $\FF_e$ of $\EE$. Moreover,
\begin{equation}
\label{eq1.4} \EE(T_ku,T_ku)\le
k(\|f(\cdot,0)\|_{L^1(E;m)}+2\|\mu\|_{TV}),
\end{equation}
where $\|\mu\|_{TV}$  stands for the total variation norm of
$\mu$.

We are mainly concerned  with equations (\ref{eq1.1}) with $L$
corresponding to a Dirichlet form. It appears, however, that a
slight modification of the proof of the main existence result from
Section \ref{sec3} yields the existence of a probabilistic
solution of (\ref{eq1.1}) in the case where $L$ correspond to a
lower-bounded semi-Dirichlet form. Although for such forms general
regularity results similar to those of Section \ref{sec4} seems to
be impossible, we find it interesting that one can still define
probabilistic solutions and investigate them  by probabilistic
methods. Our results for semi-Dirichlet forms are given in Section
\ref{sec5}. For corresponding results for parabolic equations we
defer the reader to the recent paper \cite{K:JFA}.

In Section \ref{sec6}  some applications of general results of
Sections \ref{sec2}--\ref{sec5} are indicated. In the case of
Dirichlet forms we decided to describe in some detail four quite
different examples. In the first one we consider equation
(\ref{eq1.1}) with $L$ being a nonsymmetric divergence form
operator, that is, an operator associated with local non-symmetric
regular form. In the second example $L$ is a ``divergent free''
gradient perturbation of a symmetric nonlocal operator on $\BR^d$
whose model example is the $\alpha$-laplacian. In that case $L$
corresponds to a nonsymmetric nonlocal regular form.  Then we
consider a symmetric nonlocal operator on some finely open subset
$D\subset\BR^d$, which is associated with a symmetric but in
general nonregular form. In the last example we consider the
Ornstein-Uhlenbeck operator in Hilbert space, that is, an operator
associated with a local nonregular form. In each case we formulate
specific theorem on existence, uniqueness and regularity of
solutions. To our knowledge all these results are new. We also
briefly discuss the possibility of other applications of our
general results of Sections \ref{sec2}--\ref{sec4}. Finally, to
illustrate the results of Section \ref{sec5}, we consider two
examples of  equations with operators corresponding to
semi-Dirichlet forms. In the first example $L$ is a diffusion
operator with drift term, while in the second  it is the
fractional laplacian with variable exponent.

\section{Preliminaries} \label{sec2}

In Sections \ref{sec2}--\ref{sec4} we assume that $E$ is a
metrizable Lusin space, i.e. a metrizable space which is the image
of a Polish space under a continuous bijective mapping. We adjoin
an extra point $\partial$ to $E$ as an isolated point. We define
the Borel $\sigma$-algebra on $E_{\partial}:= E\cup\{\partial\}$
by putting
$\BB(E_{\partial})=\BB(E)\cup\{B\cup\{\partial\}:B\in\BB(E)\}$. We
make the convention that any function $f:E\rightarrow\bar\BR$ is
extended to $E_{\partial}$ by setting $f(\partial)=0$. Throughout
the paper $m$  is a $\sigma$-finite positive measure on $\BB(E)$.
We extend it to $\BB(E_{\partial})$ by setting
$m(\{\partial\})=0$.

\subsection{Quasi-regular Dirichlet forms and Markov processes}

We assume throughout that $(\EE,D(\EE))$ is a quasi-regular
Dirichlet form on $L^2(E;m)$ (see \cite{MOR,MR} for the
definitions).

We also assume that $(\EE,D(\EE))$ is a semi-Dirichlet form on
$L^2(E;m)$, i.e. $(\tilde\EE,D(\EE))$, where
$\tilde\EE(u,v)=\frac12(\EE(u,v)+\EE(v,u))$ for $u,v\in D(\EE)$,
is a symmetric closed form, $(\EE,D(\EE))$ satisfies the weak
sector condition and has the following contraction property: for
every $u\in D(\EE)$, $u^+\wedge1\in D(\EE)$ and
$\EE(u+u^+\wedge1,u-u^+\wedge1))\ge0$. If, in addition,
$\EE(u-u^+\wedge1,u+u^+\wedge1))\ge0$, then $(\EE,D(\EE))$ is
called a Dirichlet form. Recall that $(\EE,D(\EE))$ satisfies the
weak sector condition if there is $K>0$ such that
\[
|\EE_1(u,v)|\le K\EE_1(u,u)^{1/2}\EE_1(v,v)^{1/2},\quad u,v\in
D(\EE),
\]
where as usual, for $\alpha\ge0$ we set
$\EE_{\alpha}(u,v)=\EE(u,v)+\alpha(u,v)$ for $u,v\in D(\EE)$
($(\cdot,\cdot)$ stands for the usual inner product in
$L^2(E;m)$). Occasionally we will assume that $(\EE,D(\EE))$
satisfies the strong sector condition, i.e. there is $K>0$ such
that
\begin{equation}
\label{eq2.1} |\EE(u,v)|\le K\EE(u,u)^{1/2}\EE(v,v)^{1/2},\quad
u,v\in D(\EE).
\end{equation}

We will denote by $(G_{\alpha})_{\alpha>0}$, (resp. $(\hat
G_{\alpha})_{\alpha>0}$  the strongly continuous contraction
resolvent (resp. coresolvent) on $L^2(E;m)$ determined by
$(\EE,D(\EE))$ (see \cite[Theorem I.2.8]{MR}), and by
$(T_t)_{t>0}$ (resp. $(\hat T_t)_{t>0})$ the strongly continuous
contraction semigroup on $L^2(E;m)$ corresponding to
$(G_{\alpha})_{\alpha>0}$ (resp. $(\hat G_{\alpha})_{\alpha>0}$).
Note that $T_t, G_{\alpha}$ and $\hat T_t, \hat G_{\alpha}$ can be
extended to a semigroup and resolvent on $L^1(E;m)$ (see
\cite[Section 1.1]{O}).

We denote by $(L,D(L))$ the generator of $(G_{\alpha})_{\alpha>0}$
(and $(T_t)_{t>0}$). By \cite[Proposition I.2.16]{MR} it can be
characterized as the unique operator on $L^2(E;m)$ such that
(\ref{eq1.2}) is satisfied.

For a closed subset $F\subset E$ we set $D(\EE)_F=\{u\in
D(\EE):u=0\mbox{ $m$-a.e. on }E\setminus F\}$. Let us recall that
an increasing sequence $\{F_k\}_{k\ge1}$ of closed subsets of $E$
is called an $\EE$-nest if $\bigcup^{\infty}_{k=1}D(\EE)_{F_k}$ is
$\tilde\EE^{1/2}$-dense in $D(\EE)$. A subset $N\subset E$ is
called $\EE$-exceptional if $N\subset\bigcap^{\infty}_{k=1}F^c_k$
for some $\EE$-nest $\{F_k\}_{k\in\BN}$. In what follows we say
that a property of points in $E$ holds $\EE$-quasi-everywhere
($\EE$-q.e. for short) if it holds outside some $\EE$-exceptional
set. An $\EE$-q.e. defined function $u$ is called
$\EE$-quasi-continuous if there exists a nest $\{F_k\}_{k\in\BN}$
such that $f\in C(\{F_k\})$, where
\[
C(\{F_k\})=\{f:A\rightarrow\BR:\bigcup^{\infty}_{k=1}F_k\subset
A\subset E, f_{|F_k} \mbox{ is continuous for every }k\in\BN\}.
\]

The notions of $\EE$-nest and  $\EE$-exceptional set can be
characterized by certain capacities  relative to $(\EE,D(\EE))$.
To formulate this characterization, fix $\varphi\in L^2(E;m)$ such
that $0<\varphi\le 1$ $m$-a.e. and for open $U\subset E$  set
\[
\mbox{Cap}_{\varphi}(U)=\inf\{\EE_1(u,u):u\in D(\EE),u\ge \tilde
G_1\varphi\,\,m\mbox{-a.e. on }U\},
\]
where $\{\tilde G_{\alpha}\}$ is the resolvent associated with
$(\tilde\EE,D(\EE))$. For arbitrary $A\subset E$ we set
\begin{equation}
\label{eq2.4}
\mbox{Cap}_{\varphi}(A)=\inf\{\mbox{Cap}_{\varphi}(U):A\subset
U\subset E,\,U\mbox{ open}\}.
\end{equation}
Then by \cite[Theorem III.2.11]{MR} an increasing sequence
$\{F_k\}_{k\ge1}$ of closed subsets of $E$ is an $\EE$-nest iff
$\lim_{k\rightarrow\infty}\mbox{Cap}_{\varphi}(E\setminus F_k)=0$,
and secondly, $N\subset E$ is $\EE$-exceptional iff
$\mbox{Cap}_{\varphi}(N)=0$. Notice that from the above it follows
in particular that the capacities $\mbox{Cap}_{\varphi}$ defined
for different $\varphi\in L^2(E;m)$ such that $0<\varphi\le 1$
$m$-a.e. are equivalent to each other.

A Dirichlet form $(\EE,D(\EE))$ is called transient if there is an
$m$-a.e. strictly positive and
bounded $g\in L^1(E;m)$ 
such that
\begin{equation}
\label{eq2.3} \int_E|u|g\,dm\le \EE(u,u)^{1/2},\quad u\in D(\EE).
\end{equation}

Notice that transience of a Dirichlet form depends only on its
symmetric part. It is known (see \cite[Corollary 3.5.34]{J2}) that
$(\EE,D(\EE))$ is transient iff the corresponding sub-Markovian
semigroup $(T_t)_{t\ge0}$ is transient, i.e. for all $u\in
L^1(E;m)$ such that $u\ge0$ $m$-a.e.,
\[
\lim_{N\rightarrow\infty}\int^N_0T_tu\,dt<\infty, \quad
m\mbox{-a.e.}
\]

Let $(\EE,D(\EE))$ be a Dirichlet form. The extended Dirichlet
space $\FF_e$ associated with the symmetric Dirichlet form
$(\tilde\EE,D(\EE))$ is the family of measurable functions
$u:E\rightarrow\BR$ such that $|u|<\infty$ $m$-a.e. and there
exists an $\tilde\EE$-Cauchy sequence $\{u_n\}\subset D(\EE)$ such
that $u_n\rightarrow u$ $m$-a.e. The sequence $\{u_n\}$ is called
an approximating sequence for $u\in\FF_e$.

For a Dirichlet form $(\EE,D(\EE))$ and $u\in\FF_e$ we set
$\EE(u,u)=\lim_{n\rightarrow\infty}\EE(u_n,u_n)$, where $\{u_n\}$
is an approximating sequence for $u$ (see \cite[Theorem
1.5.2]{FOT}). If moreover $\EE$ satisfies the strong sector
condition then we may extend $\EE$ to $\FF_e$ by putting
$\EE(u,v)=\lim_{n\rightarrow\infty}\EE(u_n,v_n)$ with
approximating sequences $\{u_n\}$ and $\{v_n\}$ for $u\in\FF_e$
and $v\in\FF_e$, respectively (it is easily seen that $\EE(u,v)$
is independent of the choice of the approximating sequences).
Observe that this extension satisfies the strong sector condition,
i.e. (\ref{eq2.1}) holds true for all $u,v\in\FF_e$.

If $(\EE,D(\EE))$ is transient then by \cite[Lemma 1.5.5]{FOT},
$(\FF_e,\tilde\EE)$ is a Hilbert space. Also note that if
$(\EE,D(\EE))$ is a quasi-regular Dirichlet form (see
\cite{MOR,MR} for the definition) then by \cite[Proposition
IV.3.3]{MR} each  $u\in D(\EE)$ admits a quasi-continuous
$m$-version denoted by $\tilde u$, and that $\tilde u$ is
$\EE$-q.e. unique for every $u\in D(\EE)$. If moreover
$(\EE,D(\EE))$ is transient then the last statement holds true for
$D(\EE)$ replaced by $\FF_e$ (see \cite[Remark 2.2]{KR:JMAA}).

In the remainder of this section we assume that $(\EE,D(\EE))$ is
a quasi-regular Dirichlet form on $L^2(E;m)$.

By \cite[Theorem IV.3.5]{MR} there exists an $m$-tight special
standard Markov process
$\BX=(\Omega,(\FF_t)_{t\ge0},(X_t)_{t\ge0},\zeta,(P_x)_{x\in
E\cup\{\partial\}})$  with state space $E$, life-time $\zeta$ and
cemetery state $\partial$ (see, e.g., \cite{MOR} or \cite[Section
IV.1]{MR} for precise definitions) which is properly associated
with $(\EE,D(\EE))$. Let $(p_t)_{t\ge0}$ be the transition
semigroup of $\BX$ defined by
\[
p_tf(x)=E_xf(X_t),\quad x\in E,\,t\ge0,\,f\in \BB^+(E).
\]
The  statement that $\BX$ is properly associated with
$(\EE,D(\EE))$ means that $p_tf$ is a quasi-continuous $m$-version
of $T_tf$ for every $t>0$ and $f\in\BB_b\cap L^2(E;m)$ (and hence
for every $t>0$ and $f\in L^2(E;m)$ by \cite[Exercise
IV.2.9]{MR}). Equivalently, by \cite[Proposition IV.2.8]{MR}, the
proper association means that $R_{\alpha}f$ is an
$\EE$-quasi-continuous $m$-version of $G_{\alpha}f$ for every
$\alpha>0$ and $f\in\BB_b\cap L^2(E;m)$, where
$(R_{\alpha})_{\alpha>0}$ is the resolvent of $\BX$, i.e.
\[
R_{\alpha}f(x)=E_x\int^{\infty}_0e^{-\alpha t}f(X_t)\,dt,\quad
x\in E,\,\alpha>0,\,f\in\BB^+(E).
\]
By \cite[Theorem IV.6.4]{MR} the process $\BX$ is uniquely
determined by $(\EE,D(\EE))$ in the sense that if $\BX'$ is
another process with state space $E$ properly associated with
$(\EE,D(\EE))$ then $\BX$ and $\BX'$ are $m$-equivalent, i.e.
there is $S\in\BB(E)$ such that $m(E\setminus S)=0$, $S$ is both
$\BX$-invariant and $\BX'$-invariant, and $p_tf(x)=p'_tf(x)$ for
all $x\in S$, $f\in\BB_b(E)$, $t>0$, where $(p'_t)_{t>0}$ is the
transition semigroup of $\BX'$.

\subsection{Smooth measures}

Recall that a positive measure $\mu$ on $\BB(E)$ is said to be
$\EE$-smooth (we write $\mu\in S$) if $\mu(B)=0$ for all
$\EE$-exceptional sets $B\in\BB(E)$ and there exists an $\EE$-nest
$\{F_k\}_{k\in\BN}$ of compact sets such that $\mu(F_k)<\infty$
for $k\in\BN$. A measure $\mu\in S$ is said to be of finite energy
integral (written $\mu\in S_0$) if there is $c>0$ such that
\begin{equation}
\label{eq2.11} \int_E|\tilde v(x)|\,\mu(dx)\le c\EE_1(v,v)^{1/2},
\quad v\in D(\EE).
\end{equation}
If additionally $(\EE,D(\EE))$ is transient then $\mu\in S$ is
said to be of finite 0-order energy integral (written $\mu\in
S^{(0)}_0$) if there is $c>0$ such that
\[
\int_E|\tilde v(x)|\,\mu(dx)\le c\EE(v,v)^{1/2}, \quad v\in\FF_e.
\]

If $(\EE,D(\EE))$ is regular and $E$ is a locally compact
separable metric space then the notion of smooth measures defined
above coincides with that in \cite{FOT}. 
Moreover, if $\mu$ is a positive Radon measure on $E$ such that
(\ref{eq2.11}) is satisfied for all $v\in C_0(E)\cap D(\EE)$ then
$\mu$ charges no $\EE$-exceptional set (see \cite[Remark
A.2]{MMS}) and hence $\mu\in S_0$.

By \cite[Proposition 2.18(ii)]{MOR} (or \cite[Proposition
III.3.6]{MR}) the reference measure $m$ is $\EE$-smooth. Therefore
if $f\in L^1(E;m)$ then $\mu=f\cdot m$ is bounded and smooth. A
general result on the structure of bounded smooth measures is
found in \cite{KR:JMAA}.

Let $\mu\in S_0$ and $\alpha>0$. Then from the Lax-Milgram theorem
(see, e.g., \cite[Theorem 2.7.41]{J1}) it follows that there exist
unique $U_{\alpha}\mu,\hat U_{\alpha}\mu\in D(\EE)$ such that
\[
\EE_{\alpha}(U_{\alpha}\mu,v)=\int_E\tilde
v(x)\,\mu(dx)=\EE_{\alpha}(v,\hat U_{\alpha}\mu),\quad v\in
D(\EE).
\]
Similarly, if $(\EE,D(\EE))$ satisfies the strong sector condition
and $\mu\in S^{(0)}_0$ then from the Lax-Milgram theorem applied
to the Hilbert space $(\FF_e,\tilde\EE)$, the form $\EE$ and
operator $J:\FF_e\rightarrow\BR$ defined by $J(v)=\int_E\tilde
v(x)\,\mu(dx)$ it follows that there exist unique $U\mu,\hat
U\mu\in\FF_e$ such that
\[
\EE(U\mu,v)=\int_E\tilde v(x)\,\mu(dx)=\EE(v,\hat U\mu),\quad
v\in\FF_e\,.
\]

Let $\MM_{0,b}$ denote the subset of $S$ consisting of all
measures $\mu$ such that $\|\mu\|_{TV}<\infty$, where
$\|\mu\|_{TV}$ denotes the total variation of $\mu$, and let
$\MM^+_{0,b}$ denote the subset of $\MM_{0,b}$ consisting of all
positive measures.

The  lemma follows below from the 0-order version of \cite[Theorem
2.2.4]{FOT} be the so-called transfer method.

\begin{lemma}
\label{lem2.2} Assume that $(\EE,D(\EE))$ is transient. If $\mu\in
S$ then there exists a nest $\{F_n\}$ such that
$\fch_{F_n}\cdot\mu\in S^{(0)}_0$ for each $n\in\BN$.
\end{lemma}
\begin{proof}
See \cite[Lemma 2.1]{KR:JMAA}.
\end{proof}

\begin{lemma}
\label{lem2.5} Assume that $(\EE,D(\EE))$ is transient and
satisfies the strong sector condition. If $\mu\in S^{(0)}_0$ then
$\{U_{\alpha}\mu\}$ is weakly $\EE$-convergent to $U\mu$ as
$\alpha\downarrow0$.
\end{lemma}
\begin{proof}
Let $v\in\FF_e$ and let $\{v_k\}\subset D(\EE)$ be an
approximating sequence for $v$. We have
\[
\EE(U\mu-U_{\alpha}\mu,v_k)=\alpha(U_{\alpha}\mu,v_k),\quad
\EE(G_0U_{\alpha}\mu,v_k)=(U_{\alpha}\mu,v_k).
\]
Hence $\EE(U\mu-U_{\alpha}\mu,v_k)=\EE(\alpha
G_0U_{\alpha}\mu,v_k)$. Letting $k\rightarrow\infty$ we deduce
$\EE(U\mu- U_{\alpha}\mu,v)=\EE(\alpha G_0U_{\alpha}\mu,v)$.
Consequently, $U\mu-U_{\alpha}\mu=\alpha G_0U_{\alpha}\mu$. In the
same manner we can see that $\hat U\mu-\hat
U_{\alpha}\mu=\alpha\hat G_0\hat U_{\alpha}\mu$. Hence,
\[
\EE(U\mu-U_{\alpha}\mu,\hat U\mu-\hat U_{\alpha}\mu)
=\alpha^2\EE(G_0U_{\alpha}\mu,\hat G_0\hat U_{\alpha}\mu)
=\alpha^2(G_0U_{\alpha}\mu,\hat U_{\alpha}\mu)\ge0.
\]
On the other hand,
\begin{align*}
\EE(U\mu-U_{\alpha}\mu,\hat U\mu-\hat U_{\alpha}\mu)
&=\EE(U\mu,\hat U\mu-\hat U_{\alpha}\mu) -\EE(U_{\alpha}\mu,\hat
U\mu) +\EE_{\alpha}(U_{\alpha}\mu,\hat U_{\alpha}\mu)\\
&\quad-\alpha(U_{\alpha}\mu,\hat U_{\alpha}\mu)\\
&=\EE(U\mu,\hat U\mu-\hat U_{\alpha}\mu)-\alpha(U_{\alpha}\mu,\hat
U_{\alpha}\mu)\\
&\le \EE(U\mu,\hat U\mu-\hat U_{\alpha}\mu)\\
&= \langle\mu,\widetilde{U\mu}\rangle-\langle\mu,\widetilde{\hat
U_{\alpha}\mu}\rangle.
\end{align*}
Since $\langle\mu,\widetilde{\hat U_{\alpha}\mu}\rangle
=\EE_{\alpha}(U_{\alpha}\mu,\hat U_{\alpha}\mu)=\langle
\mu,\widetilde{U_{\alpha}\mu}\rangle$, it follows from the above
that
\[
\EE(U_{\alpha}\mu,U_{\alpha}\mu)+\alpha(U_{\alpha}\mu,U_{\alpha}\mu)
=\EE_{\alpha}(U_{\alpha}\mu,U_{\alpha}\mu) =\langle\mu,
\widetilde{U_{\alpha}\mu}\rangle\le\langle\mu,\widetilde{U\mu}\rangle
\]
for $\alpha>0$. Hence $\{U_{\alpha}\mu\}_{\alpha>0}$ is $\tilde
\EE$-bounded and for each $k\in\BN$,
$\alpha(U_{\alpha}\mu,v_k)\rightarrow0$ as $\alpha\downarrow0$.
Suppose that $\{U_{\alpha}\mu\}$ converges $\tilde\EE$-weakly to
some $f\in\FF_e$ as $\alpha\downarrow0$. Since
\[
\EE(U_{\alpha}\mu,v_k)=\langle\mu, \tilde
v_k\rangle-\alpha(U_{\alpha}\mu,v_k),
\]
letting $\alpha\downarrow0$ shows that
$\EE(f,v_k)=\langle\mu,\tilde v\rangle=\EE(U\mu,v_k)$. Letting
$k\rightarrow\infty$ we get $\EE(f,v)=\EE(U\mu,v)$ for
$v\in\FF_e$. Thus $f=U\mu$.
\end{proof}

\subsection{Smooth measures and additive functionals}

Let $\BX$ be the Markov process properly associated with
$(\EE,D(\EE))$. In what follows for a Borel measure $\nu$ on $E$
we set $P_{\nu}(\cdot)=\int_EP_x(\cdot)\,\nu(dx)$, and by
$E_{\nu}$ we denote the expectation with respect to $P_{\nu}$.

By \cite[Theorem VI.2.4]{MR} there is a one-to-one correspondence
between $\EE$-smooth measures $\mu$ on $\BB(E)$ and positive
continuous additive functionals (PCAFs) $A$ of $\BX$. It is given
by the relation
\begin{equation}
\label{eq2.2}
\lim_{t\downarrow0}\frac1tE_m\int^t_0f(X_s)\,dA_s=\int_Ef\,d\mu,
\quad f\in\BB^+(E).
\end{equation}
In what follows the additive functional corresponding to $\mu$ in
the sense of (\ref{eq2.2}) will be denoted by $A^{\mu}$. In the
important case where $\mu= f\cdot m$ for some $f\in L^1(E;m)$ the
additive functional $A^{\mu}$ is given by
\[
A^{\mu}_t=\int^{t}_0f(X_s)\,ds,\quad t\ge0.
\]

The following lemma generalizes \cite[Lemma 4.3]{KR:JFA}.

\begin{lemma}
\label{lem2.1} If $A$ is a PCAF of $\BX$ such that
$E_xA_{\zeta}<\infty$ for $m$-a.e. $x\in E$ then
$u:E\rightarrow\bar\BR$ defined as
\[
u(x)=E_xA_{\zeta},\quad x\in E
\]
is $\EE$-quasi-continuous. In particular, $u$ is $\EE$-q.e.
finite.
\end{lemma}
\begin{proof}
Let $(\EE^{\#},D(\EE^{\#}))$ denote the regular extension of
$(\EE,D(\EE))$ specified by \cite[Theorem VI.1.2]{MR}. By
\cite[Theorem VI.1.6]{MR},  $\BX$ can be trivially extended to a
Hunt process $\BX^{\#}$ defined on $\Omega\cup(E^{\#}\setminus E)$
with state space $E^{\#}$ properly associated with the form
$(\EE^{\#},D(\EE^{\#}))$. Let us extend $A$ to a PCAF of
$\BX^{\#}$ by setting
\begin{equation}
\label{eqp.1} A^{\#}_{t}(\omega)=A_{t}(\omega),\, t\ge 0,\,
\omega\in \Omega,\quad A^{\#}_{t}(\omega)=0,\, t\ge 0,\, \omega\in
E^{\#}\setminus E.
\end{equation}
By the assumption and since $m^{\#}(E^{\#}\setminus E)=0$, we have
$E^{\#}_{x}A^{\#}_{\zeta^{\#}}<\infty$, $m^{\#}$-a.e. Therefore,
by \cite[Lemma 4.3]{KR:JFA}, the function
$u^{\#}(x)=E^{\#}_{x}A^{\#}_{\zeta^{\#}}$ is
$\EE^{\#}$-quasi-continuous on $E^{\#}$. By \cite[Corollary
VI.1.4]{MR}, $u^{\#}_{|E}$ is $\EE$-quasi-continuous on $E$, which
proves the first part of the lemma since
$u^{\#}_{|E}(x)=E_{x}A_{\zeta}$, $x\in E$. The second part is
immediate from the definition of quasi-continuity.
\end{proof}

\begin{lemma}
\label{lem4.1} Assume that $(\EE,D(\EE))$ is transient and
satisfies the strong sector condition. If $\mu\in S^{(0)}_0$ then
$u$ defined as
\[
u(x)=E_xA^{\mu}_{\zeta},\quad x\in E
\]
is a quasi-continuous version of $U\mu$.
\end{lemma}
\begin{proof}
By \cite[Proposition A.7]{MMS}, for every $\alpha>0$ the function
$R_{\alpha}\mu$ defined by $R_{\alpha}\mu(x)=
E_x\int^{\infty}_0e^{-\alpha t}dA^{\mu}_t$, $x\in E$, is a
quasi-continuous version of $U_{\alpha}\mu$. Therefore, by Lemma
\ref{lem2.5} and the Banach-Saks theorem, there exists sequences
$\alpha_n\downarrow0$ and $\{n_k\}$ such that the Ces\`aro mean
sequence $\{w_n=(1/n)\sum^n_{k=1}u_{n_k}\}$, where
$u_n=R_{\alpha_n}\mu$, is $\tilde\EE$-convergent to $U\mu$. On the
other hand, by the monotone convergence theorem,
$u_n(x)\rightarrow u(x)$ for $x\in E$, and hence
$w_n(x)\rightarrow u(x)$ for $x\in E$. Consequently,  $\{w_n\}$ is
an approximating sequence for $u$. Therefore $u\in\FF_e$ and
\[
\tilde\EE(u-U{\mu},u-U\mu)^{1/2}\le\lim_{n\rightarrow\infty}
(\tilde\EE(u-w_n,u-w_n)^{1/2}+
\tilde\EE(U\mu-w_n,U\mu-w_n)^{1/2})=0.
\]
Since $(\tilde\EE,\FF_e)$ is a Hilbert space, it follows that $u$
is an $m$-version of $U\mu$. To show that $u$ is quasi-continuous,
let us first note that by \cite[Proposition III.3.3]{MR} there is
an $\EE$-nest $\{F_k\}$ such that $\{u_n\}\subset C(\{F_k\})$.
Since $\EE$ is quasi-regular, there exists an $\EE$-nest $\{E_k\}$
consisting of compact sets. Write $F'_k=F_k\cap E_k$. Then
$\{F'_k\}$ is an $\EE$-nest consisting of compact sets and
$\{u_n\}\subset C(\{F'_k\})$. Since ${u_n}_{|F'_k}\nearrow
u_{|F'_k}$ as $n\rightarrow\infty$ for each $k\in\BN$, Dini's
theorem shows that $u$ is in $C(\{F'_k\})$, which is our claim.
\end{proof}

Let $S^{(0)}_{00}$ (resp. $\hat S^{(0)}_{00}$) denote the subset
of $S^{(0)}_{0}$ consisting of all measures $\nu$ such that
$\nu(E)<\infty$ and $\|U\nu\|_{\infty}<\infty$ (resp. $\|\hat
U\nu\|_{\infty}<\infty$).

\begin{lemma}
\label{lem2.7} Assume that $(\EE,D(\EE))$ is transient and
satisfies the strong sector condition. If $\mu\in S$, $\nu\in\hat
S^{(0)}_{00}$ then for any nonnegative Borel function $f$,
\begin{equation}
\label{eq2.16} E_{\nu}\int_{0}^{\zeta}f(X_{t})\,dA^{\mu}_{t}
=\langle f\cdot\mu,\widetilde{\hat U\nu}\rangle.
\end{equation}
\end{lemma}
\begin{proof}
By Lemma \ref{lem2.2} there exists a nest $\{F_{n}\}$ such that
$\mathbf{1}_{F_{n}}|f|\cdot|\mu|\in S^{(0)}_{0}$. By \cite[Theorem
A.8]{MMS}, for every $\alpha>0$ the function $x\mapsto
E_{x}\int_{0}^{\zeta}e^{-\alpha t}\mathbf{1}_{F_{n}}f(X_{t})\,
dA^{\mu}_{t}$ is a quasi-continuous version of
$U_{\alpha}(\mathbf{1}_{F_{n}} f\cdot\mu)$. Hence,
\begin{equation}
\label{eq2.14} E_{\nu}\int_{0}^{\zeta}e^{-\alpha
t}\mathbf{1}_{F_{n}} f(X_{t})\,
dA^{\mu}_{t}=\langle\widetilde{U_{\alpha}
(\mathbf{1}_{F_{n}}f\cdot\mu)}, \nu\rangle =\langle
\mathbf{1}_{F_{n}} f\cdot\mu,\widetilde{\hat
U_{\alpha}\nu}\rangle.
\end{equation}
Letting $\alpha\downarrow0$ and applying the monotone convergence
theorem to the left-hand side of (\ref{eq2.14}) and  Lemma
\ref{lem2.5} to the right-hand side of (\ref{eq2.14}), we obtain
\begin{equation}
\label{eq2.15} E_{\nu}\int_{0}^{\zeta}\mathbf{1}_{F_{n}}
f(X_{t})\, dA^{\mu}_{t}=\langle
\mathbf{1}_{F_{n}}f\cdot\mu,\widetilde{\hat U\nu}\rangle.
\end{equation}
Letting $n\rightarrow\infty$ in (\ref{eq2.15}) yields
(\ref{eq2.15}) with $F_n$ replaced by $\bigcup^{\infty}_{n=1}F_n$,
which implies (\ref{eq2.16}) because
$(\bigcup^{\infty}_{n=1}F_n)^c$ is an exceptional set.
\end{proof}

\begin{lemma}
\label{lem2.8} Assume that $(\mathcal{E},D(\mathcal{E}))$ is
transient, $\mu_{1}\in S$,  $\mu_{2}\in \MM^{+}_{0,b}$. If
\[
E_{x}\int_{0}^{\zeta} dA^{\mu_{1}}_t \le E_{x}\int_{0}^{\zeta}
dA^{\mu_{2}}_t
\]
for $m$-a.e. $x\in E$ then $\|\mu_{1}\|_{TV}\le \|\mu_{2}\|_{TV}$.
\end{lemma}
\begin{proof}
Let $(\EE^{\#},D(\EE^{\#}))$, $\mu^{\#}$ be defined as in the
proof of Lemma \ref{lem2.1}, and let $(A^{\mu})^{\#}$ be defined
by (\ref{eqp.1}) with $A$ replaced by $A^\mu$. It is an elementary
check that $(A^{\mu})^{\#}=A^{\mu^{\#}}$.  By the assumptions and
since $m^{\#}(E^{\#}\setminus E)=0$,
\[
E^{\#}_{x}\int_{0}^{\zeta^{\#}}dA^{\mu_{1}^{\#}}_t \le
E^{\#}_{x}\int_{0}^{\zeta^{\#}}dA^{\mu_{2}^{\#}}_t
\]
for $m$-a.e. $x\in E^{\#}$. Clearly $\mu_{2}^{\#}\in
\mathcal{M}_{0,b}(E^{\#})$. Therefore $\|\mu_{1}^{\#}\|_{TV}\le
\|\mu_{2}^{\#}\|_{TV}$ by \cite[Lemma 5.4]{KR:JFA}, and hence
$\|\mu_{1}\|_{TV}\le \|\mu_{2}\|_{TV}$.
\end{proof}
\medskip

The following lemma is probably known, but we do not have a
reference.
\begin{lemma}
\label{lem2.10} Assume that $(\EE,D(\EE))$ is transient and
satisfies the strong sector condition. Let $B\in\BB(E)$. If
$\nu(B)=0$ for every $\nu\in S^{(0)}_{00}$ then $B$ is
$\EE$-exceptional.
\end{lemma}
\begin{proof}
Let $(\EE^{\#},D(\EE^{\#}))$ be the regular extension of
$(\tilde\EE,D(\EE))$ specified in \cite[Theorem VI.1.2]{MR}. Let
$\nu^{\#}$ be  a smooth measure on $\BB(\EE^{\#})$ and let
$\nu={\nu^{\#}}_{|\BB(E)}$. If $A\in\BB(E)$ is $\EE$-exceptional
then by \cite[Corollary VI.1.4]{MR}, $A$ is
$\EE^{\#}$-exceptional, and hence $\nu(A)=\nu^{\#}(A)=0$.
Moreover, if $\{F_k\}$ is a nest in $E^{\#}$ such that
$\nu^{\#}(F_k)<\infty$ for $k\in\BN$ and $\{E_k\}$ is a nest in
$E$ as in \cite[Theorem VI.1.2]{MR}, then $\{F_k\cap E_k\}$ is an
$\EE$-nest of compact sets in $E$ such that $\nu(F_k\cap
E_k)<\infty$, $k\in\BN$. Thus $\nu$ is a smooth measure on
$\BB(E)$. If moreover $\nu^{\#}\in S^{(0)}_{00}(E^{\#})$ then
$\nu(E)<\infty$ and, for $\eta\in D(\EE)$,
\[
\langle\nu,\tilde\eta\rangle=\langle\nu^{\#},\tilde\eta\rangle\le
c\EE^{\#}(\eta,\eta)^{1/2}=c\EE(\eta,\eta)^{1/2}.
\]
From this in the same manner as in the proof of Lemma \ref{lem2.2}
one can deduce that $\langle\nu,\tilde\eta\rangle\le
c\EE(\eta,\eta)^{1/2}$ for $\eta\in\FF_e$, i.e. that $\nu\in
S^{(0)}_0$. From Lemma \ref{lem4.1} and the fact that
$A^{\nu^{\#}}=(A^{\nu})^{\#}$ it  follows now that
${U\nu^{\#}}_{|E}$ is an $m$-version of $U\nu$. Therefore
$\|U\nu\|_{\infty}<\infty$, which proves that $\nu\in
S^{(0)}_{00}$. Hence $\nu(B)=0$, and consequently
$\nu^{\#}(B)=\nu(B)=0$ for every $\nu^{\#}$ in
$S^{(0)}_{00}(E^{\#})$. Therefore from the 0-order version of
\cite[Theorem 2.2.3]{FOT} (see remark following \cite[Corollary
2.2.2]{FOT}) we conclude that $\mbox{Cap}^{\#}_{1,1}(B)=0$, where
$\mbox{Cap}^{\#}_{1,1}$ denotes the capacity relative to
$(\EE^{\#},D(\EE^{\#}))$ defined in \cite[Definition III.2.4]{MR}
(see also \cite[Exercise III.2.10]{MR}). Hence
$\mbox{Cap}^{\#}_{\varphi}(B)=0$ by \cite[Proposition VI.1.5]{MR},
and consequently $\mbox{Cap}_{\varphi}(B)=0$ by \cite[Corollary
VI.1.4]{MR} ($\mbox{Cap}_{\varphi}$ is defined by (\ref{eq2.4})).
By remark following (\ref{eq2.4})  this implies that $B$ is
$\EE$-exceptional.
\end{proof}

\section{Probabilistic solutions} \label{sec3}

In this section we assume that $(\EE,D(\EE))$ is a quasi-regular
Dirichlet form on $L^2(E;m)$. We will need the following
assumptions on $f$ from the right-hand side of (\ref{eq1.1}):
\begin{enumerate}
\item[(A1)]$f:E\times\BR\rightarrow\BR$ is measurable and
$y\mapsto f(x,y)$ is continuous for every $x\in E$,
\item[(A2)]$(f(x,y_{1})-f(x,y_{2}))(y_{1}-y_{2})\le 0$ for all
$y_{1}, y_{2}\in \BR$ and $x\in E$,
\item[(A3)]$f(\cdot,y)\in L^1(E;m)$ for every $y\in\BR$,
\item[(A4)]$\mu\in\MM_{0,b}$,
\end{enumerate}
and
\begin{enumerate}
\item[(A3${}^*$)]for every $y\in\BR$ the function $f(\cdot,y)$
is quasi-$L^1$ with respect to $(\EE,D(\EE))$, i.e. $t\mapsto
f(X_t,y)$ belongs to $L^1_{loc}(\BR_{+})$ $P_x$-a.s. for q.e.
$x\in E$,
\item[(A4${}^*$)]$E_x\int^{\zeta}_0|f(X_t,0)|\,dt<\infty$,
$E_x\int^{\zeta}_0d|A^{\mu}|_t<\infty$ for q.e. $x\in E$.
\end{enumerate}

Note that  in our previous paper  \cite{KR:JFA} devoted to
equations of the form (\ref{eq1.1}) we followed \cite{BBGGPV} in
assuming that $f$ satisfies (A1), (A2), (A4) and the following
condition: for every $r>0$, $F_r\in L^1(E;m)$, where
$F_r(x)=\sup_{|y|\le r}|f(x,y)|$. Obviously  (A3) is weaker than
the last condition. Likewise, (A3${}^*$) is weaker than the
corresponding condition (A3$'$) in \cite{KR:JFA} saying that for
every $r>0$ the function $t\mapsto F_r(X_t,y)$ belongs to
$L^1_{loc}(\BR_+)$ $P_x$-a.s. for q.e. $x\in E$. Observe, however,
that (A3) together with (A1), (A2)  imply that $F_r\in L^1(E;m)$.
Likewise, (A3${}^*$) together with (A1), (A2), imply condition
(A3$'$) from \cite{KR:JFA}.

Define the co-potential operator as
\[
\hat G\phi=\lim_{n\rightarrow\infty}\hat G_{1/n}\phi,\quad \phi\in
L^1(E;m),\,\phi\ge0
\]
and for  $\mu\in S$ set
\[
R\mu(x)=E_x\int^{\zeta}_0dA^{\mu}_t,\quad x\in E.
\]
\begin{lemma}
If $(\EE,D(\EE))$ is transient then for any $\mu\in S$ and
$\phi\in L^1(E;m)$ such that $\phi\ge0$ we have
\begin{equation}
\label{eq3.01} (R\mu,\phi)=\langle\mu,\widetilde{\hat
G\phi}\rangle.
\end{equation}
\end{lemma}
\begin{proof}
By Lemma \ref{lem2.2} there is a nest $\{F_n\}$  such that
$\fch_{F_n}\cdot\mu\in S^{(0)}_0$ for each $n\in\BN$. Let
\[
R_{\alpha}(\fch_{F_n}\cdot\mu)(x)=E_x\int^{\zeta}_0e^{-\alpha
t}\fch_{F_n}(X_t)\,dA^{\mu}_t,\quad \alpha>0,\quad x\in E.
\]
Since $R_{\alpha}(\fch_{F_n}\cdot\mu)$ is an $m$-version of
$U_{\alpha}(\fch_{F_n}\cdot\mu)$, for any nonnegative $\phi$ in
$L^1(E;m)\cap L^2(E;m)$ we have
\[
\langle\fch_{F_n}\cdot\mu,\widetilde{\hat G_{\alpha}\phi}\rangle
=\EE_{\alpha}(U_{\alpha}(\fch_{F_n}\cdot\mu),\hat
U_{\alpha}\phi)=\EE_{\alpha}(R_{\alpha}(\fch_{F_n}\cdot\mu),\hat
U_{\alpha}\phi).
\]
Hence,
\begin{equation}
\label{eq3.02} \langle\fch_{F_n}\cdot\mu,\widetilde{\hat
G_{\alpha}\phi}\rangle=(R_{\alpha}(\fch_{F_n}\cdot\mu),\phi).
\end{equation}
In fact, approximating $\phi\in L^1(E;m)$ by a sequence
$\{\phi_k\}\subset L^1(E;m)\cap L^2(E;m)$ such that
$0\le\phi_k\nearrow\phi$ yields (\ref{eq3.02}) for any $\phi\in
L^1(E;m)$ such that $\phi\ge0$. Finally, letting
$\alpha\downarrow0$ and then $n\rightarrow\infty$ in
(\ref{eq3.02}) gives (\ref{eq3.01}).
\end{proof}
\medskip

Let $\RR$ be defined by (\ref{eq1.03}). If $\mu$ is smooth and
$R|\mu|<\infty$ $m$-a.e. then from (\ref{eq3.01}) and the fact
that $m$ is $\sigma$-finite it follows that $\mu\in\mathcal{R}$.
Furthermore, if $\mu\in\mathcal{R}$ then by (\ref{eq3.01}),
$R|\mu|<\infty$ $m$-a.e. Thus $\mathcal{R}$ can be equivalently
defined as
\[
\RR=\{\mbox{$\mu:\mu$ is smooth, $R|\mu|<\infty$, $m$-a.e.}\}.
\]
It follows in particular that (A4${}^*$) is satisfied iff
$f(\cdot,0)\cdot m\in\mathcal{R}$ and $\mu\in\mathcal{R}$.

\begin{proposition}
\label{prop2.2} If $(\EE,D(\EE))$ is transient then
$\MM_{0,b}\subset\RR$.
\end{proposition}
\begin{proof}
Apply \cite[Corollary 1.3.6]{O} to the dual form
$(\hat\EE,D(\EE))$.
\end{proof}
\medskip

In general the inclusion in Proposition \ref{prop2.2} is strict.
To see this let us consider the classical form
\begin{equation}
\label{eq3.03} \BD(u,v)=\frac12\int_D\langle\nabla u,\nabla
v\rangle_{\BR^d}\,dx,\quad u,v\in H^1_0(D)
\end{equation}
on $L^2(D;dx)$, where $D$ is a bounded open subset of $\BR^d$. If
$d\ge3$ and $D$ has smooth boundary then $R1$ is a continuous
strictly positive function such that $R1(x)\approx\delta (x)$ for
$x\in D$, where $\delta(x)=\mbox{dist}(x,\partial D)$ (for the
last property see \cite[Proposition 4.9]{Ku}). Since $R1$ is an
$m$-version of $G1=\hat G1$, it follows that
$L^1(D;\delta(x)\,dx)\in\mathcal{R}$, so $\mathcal{R}$ contains
positive Radon measures of infinite total variation. Elliptic and
parabolic equations with right-hand side in $L^q(D;\delta(x)\,dx)$
$(q\ge1)$  are studied for instance in \cite{FSW}.

\begin{remark}
Assume that $(\EE,D(\EE))$ is transient. Then by Lemma
\ref{lem2.1} and Proposition \ref{prop2.2},  (A3) implies
(A3${}^*$) and (A4) implies (A4${}^*$).
\end{remark}

\subsection{BSDEs}

Let $(\Omega,(\FF_t)_{t\ge0},P)$  be  a filtered probability
space. We will need the following classes of processes defined on
$\Omega$.

$\DD$ is the space of all $(\FF_t)$-progressively measurable
c\`adl\`ag processes, and $\DD^q$, $q>0$, is the subspace of $\DD$
consisting of all processes $Y$ such that
$E\sup_{t\ge0}|Y_t|^q<\infty$.

$\MM$ (resp. $\MM_{loc}$) is the space of all c\`adl\`ag
$((\FF_t),P)$-martingales (resp. local martingales) $M$ such that
$M_0=0$. $\MM^2$ is the subspace of $\MM$ consisting of all
martingales such that $E[M]_{\infty}<\infty$.

We  say that a c\`adl\`ag $(\FF_t)$-adapted process $Y$ is of
Doob's class (D)  if the collection $\{Y_{\tau},\tau\in\TT\}$,
where $\TT$ is the set a finite valued $(\FF_t)$-stopping times,
is uniformly integrable. For a process $Y$ of class (D)  we set
$\|Y\|_{1}=\sup\{E|Y_{\tau}|,\tau\in\TT\}$.

In the present subsection $\xi$ is an $\FF_T$-measurable random
variable, $\zeta$ is an $(\FF_t)$-stoping time, $V$ is a
continuous $(\FF_t)$-adapted  finite variation process  such that
$V_0=0$ and $f:[0,\infty)\times\Omega\times\BR\rightarrow\BR$ is a
measurable function  such that $f(\cdot,y)$ is
$(\FF_t)$-progressively measurable process for every $y\in\BR$
(for brevity in notation we omit the dependence of $f$ on
$\omega$).

\begin{definition}
We say that  a pair $(Y,M)$ of processes  is a solution of the
backward stochastic differential equation on $[0,T]$ with terminal
condition $\xi$ and coefficient $f+dV$ (BSDE$(\xi,f+dV)$ for
short) if
\begin{enumerate}
\item[(a)] $Y\in\DD$, $Y$ is of class (D)  and $M\in\MM_{loc}$,
\item[(b)] the mapping $[0,T]\ni t\mapsto f(t,Y_t)$ belongs to
$L^1(0,T)$ $P$-a.s. and
\[
Y_t=\xi+\int^{T}_{t}f(r,Y_r)\,dr +\int^{T}_{t}dV_r
-\int^{T}_{t}dM_r,\quad t\in[0,T],\quad P\mbox{-a.s.}
\]
\end{enumerate}
\end{definition}

\begin{definition}
We say that  a pair $(Y,M)$ is a solution of the backward
stochastic differential equation with terminal condition 0 at
terminal time $\zeta$ and coefficient $f+dV$
(BSDE${}^{\zeta}(f+dV)$ for short) if
\begin{enumerate}
\item[(a)]$Y\in\DD$, $Y_{t\wedge\zeta}\rightarrow0$ $P$-a.s. as
$t\rightarrow\infty$, $Y$ is of class (D)  and $M\in\MM_{loc}$,
\item[(b)]for every $T>0$, $[0,T]\ni t\mapsto f(t,Y_t)$ belongs
to $L^1(0,T)$ $P$-a.s. and
\[
Y_t=Y_{T\wedge\zeta}+\int^{T\wedge\zeta}_{t\wedge\zeta}f(r,Y_r)\,dr
+\int^{T\wedge\zeta}_{t\wedge\zeta}dV_r
-\int^{T\wedge\zeta}_{t\wedge\zeta}dM_r,\quad t\in[0,T],\quad
P\mbox{-a.s.}
\]
\end{enumerate}
\end{definition}

Let us consider the following hypotheses:
\begin{enumerate}
\item [(H1)] For every $t\in [0,T]$ the function
$\BR\ni y\mapsto f(t,y)$ is continuous $P$-a.s.
\item [(H2)] For every $t\in [0,T]$ the function $\BR\ni y\mapsto f(t,y)$
is $P$-a.s. nondecreasing.
\item [(H3)] For every $y\in\BR$ the function
$[0,T]\ni t\mapsto f(t,y)$ belongs to $ L^1(0,T)$ $P$-a.s.
\end{enumerate}

\begin{remark}
\label{uw3.100} The following Theorems \ref{th3.0} and \ref{th3.1}
were stated in \cite{KR:JFA} (see Theorems 2.7 and 3.4 there).
Unfortunately, there are some gaps in the proofs of these results
in \cite{KR:JFA}.  Namely, in the proof of \cite[Theorem
2.7]{KR:JFA} we applied Lemma \cite[Lemma 2.6]{KR:JFA}, which is
true, but its proof is correct under the additional assumption
that the coefficient $f$ is bounded from below by some linear
function  with respect to $y$ (otherwise the function $f_n$
appearing in the proof is not well defined). Secondly, in the
proof of \cite[Theorem 3.4]{KR:JFA} we applied \cite[Lemma
2.5]{KR:JFA}, which is correct only for $p\ge2$ or under the
additional assumption that the solution $(Y,M)$ is continuous (the
reason for this is that in the proof of \cite[Lemma 2.5]{KR:JFA}
we used the Burkholder-Davis-Gundy inequality with exponent
$p/2$). Here we give the proofs of \cite[Theorem 2.7]{KR:JFA} and
\cite[Theorem 3.4]{KR:JFA} in full generality.
\end{remark}

In what follows we denote by $T_c$, $c\ge0$, the truncation
operator, i.e.
\begin{equation}
\label{eq3.15} T_c(x)=(-c)\vee x\wedge c,\quad x\in E.
\end{equation}

\begin{lemma}
\label{lm3.111} Assume that \mbox{\rm{(H1)--(H3)}} are satisfied
and there exists $c>0$ such that
\[
T\cdot\sup_{0\le t\le T}|f(t,0)|+|V|_T+|\xi|\le c.
\]
Then there exists a unique solution $(Y,M)\in \DD^2\otimes\MM^2$
of BSDE$(\xi,f+dV)$.
\end{lemma}
\begin{proof}
Let $f_c(t,y)=f(t,T_c(y))$. Then  $|\inf_{y\in\BR}
f_c(t,y)|<\infty$ and the proof of \cite[Lemma 2.6]{KR:JFA} shows
(see Remark \ref{uw3.100}) that  there exists a unique solution
$(Y,M)\in \DD^2\otimes\MM^2$ of BSDE$(\xi,f_c+dV)$. But by the
Tanaka-Meyer formula and the assumptions,
\begin{align*}
|Y_t|&\le E\Big(|\xi|+\int_t^T\mbox{sgn}(Y_r)f_c(r,Y_r)\,dr
+\int_t^T\mbox{sgn}(Y_r)\,dV_r|\FF_t\Big)\\& \le
E\Big(|\xi|+\int_0^T|f(r,0)|\,dr+\int_0^T\,d|V|_r|\FF_t\Big)\le c,
\end{align*}
so in fact $(Y,M)$ is a solution of BSDE$(\xi,f+dV)$.
\end{proof}

\begin{theorem}
\label{th3.0} Assume that \mbox{\rm{(H1)--(H3)}} are satisfied and
\[
E\Big(|\xi|+\int_0^T|f(t,0)|\,dt+\int_0^Td|V|_t\Big)<\infty.
\]
Then there exists  a solution $(Y,M)$ of
\mbox{\rm{BSDE}}$(\xi,f+dV)$ such that $Y\in\DD^q$ for every $q\in
(0,1)$ and $M$ is a uniformly integrable martingale.
\end{theorem}
\begin{proof}
Let us put
\[
\xi^n=T_n(\xi),\quad f_n(t,y)=f(t,y)-f(t,0)+T_n(f(t,0)),\quad V^n_t
=\int_0^t\mathbf{1}_{\{|V|_r\le n\}}\,dV_r.
\]
By Lemma \ref{lm3.111}, for every $n\ge 1$ there exists a solution
$(Y^n,M^n)$ of BSDE$(\xi^n,f^n+dV^n)$. As in the proof of
\cite[Theorem 2.7]{KR:JFA} we show that there exists a process $Y$
of class (D) such that $Y\in\DD^q$ for $q\in(0,1)$ and
\begin{equation}
\label{eq3.100} E\sup_{0\le t\le T} |Y^n_t-Y_t|^q\rightarrow 0
\end{equation}
for every $q\in (0,1)$. By the Tanaka-Meyer formula and (H2),
\[
|Y^n_t|\le E\Big(|\xi|+\int_0^T|f(r,0)|\,dr
+\int_0^T\,d|V|_r\big|\FF_t\Big),\quad t\in[0,T].
\]
Let $R$ denote a c\`adl\`ag process such that for every $t\in
[0,T]$, $R_t$ is equal to the right-hand side of the above
inequality. Then
\[
|Y^n_t|\le R_t,\quad t\in [0,T],\quad n\ge 1.
\]
For $k,N\in\BN$ set
\begin{equation}
\label{eq3.107} \tau_{k,N}=\inf\{t\ge 0: R_t\ge k\mbox{ or }
\int_0^t(|f(r,-k)|+|f(r,k)|)\,dr\ge N\}\wedge T.
\end{equation}
By the definition of a solution of BSDE$(\xi^n,f_n+dV^n)$,
\begin{align}
\label{eq3.102} Y^n_{t\wedge\tau_{k,N}}=E\Big(Y^n_{\tau_{k,N}}
+\int_{t\wedge\tau_{k,N}}^{\tau_{k,N}}f_n(r,Y^n_r)\,dr
+\int_{t\wedge\tau_{k,N}}^{\tau_{k,N}}\,dV^n_r\big|\FF_t\Big).
\end{align}
From the definition of $\tau_{k,N}$ it follows that
\[
|\int_{t\wedge\tau_{k,N}}^{\tau_{k,N}}f_n(r,Y^n_r)\,dr|
\le\int_0^{\tau_{k,N}}|f(r,Y^n_r)|\,dr\le N.
\]
From this, (H1) and (\ref{eq3.100}) we conclude that
\[
E\int_0^{\tau_{k,N}}|f_n(t,Y_t^n)-f(t,Y_t)|\,dt\rightarrow0
\]
as $n\rightarrow\infty$. Therefore letting $n\rightarrow\infty$ in
(\ref{eq3.102}) and using (\ref{eq3.100}) and Doob's maximal
inequality (for details see  the argument following (\ref{eq6.3}))
we obtain
\begin{equation}
\label{eq3.109} Y_{t\wedge\tau_{k,N}}
=E\Big(Y_{\tau_{k,N}}+\int_{t\wedge\tau_{k,N}}^{\tau_{k,N}}f(r,Y_r)\,dr
+\int_{t\wedge\tau_{k,N}}^{\tau_{k,N}}\,dV_r\big|\FF_t\Big).
\end{equation}
By \cite[Lemma 2.3]{KR:JFA},
\[
E\int_0^T|f_n(t,Y^n_t)|\,dt\le
E\Big(|\xi|+\int_0^T|f(t,0)|\,dt+\int_0^T\,d|V|_t\Big),\quad
n\ge1,
\]
so applying Fatou's lemma and (\ref{eq3.100}) gives
\begin{equation}
\label{eq3.101} E\int_0^T|f(t,Y_t)|\,dt<\infty.
\end{equation}
By (H3), $\tau_{k,N}\rightarrow\tau_k$ $P$-a.s. as
$N\rightarrow\infty$, where
\begin{equation}
\label{eq3.110} \tau_k=\inf\{t\ge0:R_t\ge k\}\wedge T.
\end{equation}
Therefore letting $N\rightarrow\infty$ in (\ref{eq3.109})  and
using (\ref{eq3.101}) and the fact that $Y$ is of class (D) we get
\begin{equation}
\label{eq3.104}
Y_{t\wedge\tau_k}=E\Big(Y_{\tau_k}+\int_{t\wedge\tau_k}^{\tau_k}
f(r,Y_r)\,dr+\int_{t\wedge\tau_k}^{\tau_k}\,dV_r\big|\FF_t\Big).
\end{equation}
Since $R$ is a c\`adl\`ag process, $\tau_k\rightarrow T$ $P$-a.s.
as $k\rightarrow\infty$. Therefore letting $k\rightarrow\infty$ in
(\ref{eq3.104}) and using once again  (\ref{eq3.101}) and the fact
that $Y$ is of class (D) gives
\[
Y_t=E\Big(\xi+\int_t^T f(r,Y_r)\,dr+\int_t^T\,dV_r\big|\FF_t\Big).
\]
It follows that the pair $(Y,M)$, where $M$ is a c\`adl\`ag
process such that
\[
M_t=E\Big(\xi+\int_0^T
f(r,Y_r)\,dr+\int_0^T\,dV_r\big|\FF_t\Big)-Y_0,\quad t\in[0,T],
\]
is a solution of BSDE$(\xi,f+dV)$.
\end{proof}

\begin{theorem}
\label{th3.1} Assume that \mbox{\rm(H1)--(H3)} are satisfied for
every $T>0$, and that
\[
E\Big(\int_0^\zeta |f(t,0)|\,dt+\int_0^\zeta \,d|V|_t\Big)<\infty.
\]
Then there exists a unique solution $(Y,M)$ of
BSDE${}^{\zeta}(f+dV)$. Moreover, $Y\in\DD^{q}$ for $q\in(0,1)$,
$M$ is a uniformly integrable $(\FF_t)$-martingale and
\begin{equation}
\label{eq3.108} E\int_0^\zeta|f(t,Y_t)|\,dt\le E\Big(\int_0^\zeta
|f(t,0)|\,dt+\int_0^\zeta \,d|V|_t\Big).
\end{equation}
\end{theorem}
\begin{proof}
The uniqueness part is a direct consequence of \cite[Corollary
3.2]{KR:JFA}. To prove the existence we slightly modify the proof
of \cite[Theorem 3.4]{KR:JFA}. By Theorem \ref{th3.0}, for each
$n\in\mathbb{N}$ there exists a unique solution $(Y^{n},M^{n})$ of
the BSDE$(0,\fch_{[0,\zeta]}f+dV_{\cdot\wedge\zeta})$ on $[0,n]$
such that $Y^n\in\DD^{q}$ for $q\in (0,1)$ and $M^n$ is a
uniformly integrable $(\FF_t)$-martingale. By the definition of a
solution,
\begin{equation}
\label{eq6.1}
Y^{n}_{t}=\int_{t}^{n}\mathbf{1}_{[0,\zeta]}(r)f(r,Y^{n}_{r})\,dr
+\int_{t}^{n}dV_{r\wedge\zeta}-\int_{t}^{n}dM^{n}_{r}, \quad t\in
[0,n].
\end{equation}
Set $(Y^{n}_{t},M^{n}_{t})=(0,M^{n}_{n})$ for $t\ge n$. Then as in
the proof of \cite[Theorem 3.4]{KR:JFA} we show (see the proof of
\cite[(3.11)]{KR:JFA} and the inequality following it) that for
every $m>n$ and $q\in(0,1)$,
\[
E\sup_{t\ge 0}|Y^m_{t}-Y^n_t|^{q}\le
\frac{1}{1-q}\Big(E\int_{n\wedge\zeta}^{\zeta}|f(r,0)|\,dr
+\int_{n\wedge\zeta}^{\zeta}d|V|_{r}\Big)^{q}
\]
and
\[
\|Y^m-Y^n\|_{1}\le
\frac{1}{1-q}\Big(E\int_{n\wedge\zeta}^{\zeta}|f(r,0)|\,dr
+\int_{n\wedge\zeta}^{\zeta} d|V|_{r}\Big)^{q}.
\]
Therefore there exists $Y$ such that $Y\in\DD^{q}$ for $q\in
(0,1)$, $Y$ is of class (D) and $Y^{n}\rightarrow Y$ in the norm
$\|\cdot\|_{1}$ and in $\DD^{q}$ for  $q\in (0,1)$. Since
$Y^n_{\zeta}=0$ $P$-a.s. for $n\in\BN$, from the latter
convergence it follows in particular that
$Y_{t\wedge\zeta}\rightarrow 0$ as $t\rightarrow \infty$. In much
the same way as in the proof of \cite[(3.5)]{KR:JFA} we show that
\[
|Y^n_t|\le
E\Big(\int^{n\wedge\zeta}_{t\wedge\zeta}\mbox{sgn}(Y^n_{r-})
(f(r,Y^n_r)\,dr+dV_r)\big|\FF_t\Big).
\]
From this and (H2) we get
\begin{align}
\label{eq3.106} |Y^n_t|&\le E \Big(\int^{\zeta}_{t\wedge\zeta}
(|f(r,0)|\,dr+d|V|_r)\big|\FF_t\Big)\\
&\le E \Big(\int^{\zeta}_{0} (|f(r,0)|\,dr+d|V|_r)\big|\FF_t\Big)
=: R_t,\quad t\ge0. \nonumber
\end{align}
Let $\tau_{k,N}$ be defined by (\ref{eq3.107}) but with $R_t$ from
(\ref{eq3.106}). By (\ref{eq6.1}), for $T<n$ we have
\begin{equation}
\label{eq6.3} Y^n_{t\wedge\tau_{k,N}}
=E\Big(Y^n_{\zeta\wedge\tau_{k,N}} +\int^{\zeta\wedge\tau_{k,N}}
_{t\wedge\zeta\wedge\tau_{k,N}} (f(r,Y^n_r)\,dr
+dV_r)\big|\FF_t\Big),\quad t\in[0,T],\quad P\mbox{-a.s.}
\end{equation}
By Doob's maximal inequality, for every $\varepsilon>0$,
\[
\lim_{n\rightarrow\infty}P(\sup_{t\le T}
|E(Y^n_{\zeta\wedge\tau_{k,N}} -Y_{\zeta\wedge\tau_{k,N}}
\big|\FF_t)|>\varepsilon)  \le
\varepsilon^{-1}\lim_{n\rightarrow\infty}
E|Y^n_{\zeta\wedge\tau_{k,N}} -Y_{\zeta\wedge\tau_{k,N}}|=0.
\]
Since $\sup_{t\ge0}|Y^n_t-Y_t|\rightarrow0$ in probability $P$, it
follows from the definition of $\tau_{k,N}$ that
\begin{align*}
&\lim_{n\rightarrow\infty}E\int^{\zeta\wedge\tau_{k,N}}_0
|f(r,Y^n_r)-f(r,Y_r)|\,dr\\
&\qquad=\lim_{n\rightarrow\infty} E\int^{\zeta\wedge\tau_{k,N}-}_0
|f(r,Y^n_r)-f(r,Y_r)|\,dr=0.
\end{align*}
Applying Doob's maximal inequality we conclude from the above that
\[
\lim_{n\rightarrow\infty}P\Big(\sup_{t\le T}
E\Big(\int^{\zeta\wedge\tau_{k,N}}_t
|f(r,Y^n_r)-f(r,Y_r)|\,dr\big|\FF_t\Big)>\varepsilon\Big)=0
\]
for $\varepsilon>0$. Therefore letting $n\rightarrow\infty$ in
(\ref{eq6.3}) we can assert that $P$-a.s. we have
\begin{equation}
\label{eq6.4} Y_{t\wedge\tau_{k,N}}
=E\Big(Y_{\zeta\wedge\tau_{k,N}}
+\int^{\zeta\wedge\tau_{k,N}}_{t\wedge\zeta\wedge\tau_{k,N}}
(f(r,Y_r)\,dr+dV_r) \big|\FF_t\Big),\quad t\in[0,T].
\end{equation}
By (H3), $\tau_{k,N}\rightarrow\tau_k$ $P$-a.s. as
$N\rightarrow\infty$, where $\tau_k$ is defined by (\ref{eq3.110})
but with $R_t$ defined by (\ref{eq3.106}). Hence
$Y_{\zeta\wedge\tau_{k,N}}\rightarrow Y_{\zeta\wedge\tau_k}$
$P$-a.s. as $N\rightarrow\infty$, and consequently
$E|Y_{\zeta\wedge\tau_{k,N}}- Y_{\zeta\wedge\tau_k}|\rightarrow0$
since $Y$ is of class (D). Therefore letting $N\rightarrow\infty$
in (\ref{eq6.4}) we obtain
\begin{equation}
\label{eq3.111} Y_{t\wedge\tau_{k}} =E\Big(Y_{\zeta\wedge\tau_{k}}
+\int^{\zeta\wedge\tau_{k}}_{t\wedge\zeta\wedge\tau_{k}}
(f(r,Y_r)\,dr+dV_r) \big|\FF_t\Big),\quad t\in[0,T].
\end{equation}
Since we may assume that $R$ is a c\`adl\`ag process,
$\tau_{k}\rightarrow T$, $P$-a.s. as $k\rightarrow\infty$. Hence
$Y_{\zeta\wedge\tau_k}\rightarrow Y_{T\wedge\zeta}$ $P$-a.s. as
$k\rightarrow\infty$, and consequently $E|Y_{\zeta\wedge\tau_k}-
Y_{T\wedge\zeta}|\rightarrow0$ since $Y$ is of class (D). Also,
$E|Y_{T\wedge\zeta}|\rightarrow0$ as $T\rightarrow\infty$ since we
know that $Y_{T\wedge\zeta}\rightarrow0$ $P$-a.s. By \cite[Lemma
2.3]{KR:JFA}, for every $n\ge1$,
\[
E\int_0^{n\wedge\zeta}|f(r,Y^n_r)|\,dr\le
E\Big(|Y^n_{n\wedge\zeta}|+\int_0^{n\wedge\zeta}|f(r,0)|\,dr
+\int_0^{n\wedge\zeta}\,d|V|_r\Big).
\]
Letting $n\rightarrow\infty$ in the above inequality and applying
Fatou's lemma and the first inequality in (\ref{eq3.106}) we get
(\ref{eq3.108}). Therefore letting $k\rightarrow\infty$ in
(\ref{eq3.111}) and then letting $T\rightarrow\infty$ and using
Doob's maximal inequality we obtain
\[
Y_t=E\Big(\int^{\zeta}_{t\wedge\zeta} (f(r,Y_r)\,dr
+dV_r)|\FF_t\Big),\quad t\ge0,\quad P\mbox{-a.s.}
\]
From this, one can easily deduce that the pair $(Y,M)$, where
\[
M_{t}=E\Big(\int_{0}^{\zeta} f(r,Y_r)\,dr +\int_{0}^{\zeta}
dV_{r}|\FF_{t}\Big)-Y_{0},\quad t\ge 0,
\]
is a solution of BSDE${}^{\zeta}(f+dV)$. Finally, since the
martingale $M$ is closed, it is uniformly integrable.
\end{proof}

\subsection{Existence and uniqueness of probabilistic solutions}

Let $(L,D(L))$ be the operator defined by  (\ref{eq1.2}).

\begin{definition}
Let $\mu\in\RR$. We say that an $\EE$-quasi-continuous
$u:E\rightarrow\BR$ is a probabilistic solution of the equation
\begin{equation}
\label{eq3.1} -Lu=f_u+\mu,
\end{equation}
where $f_u(x)=f(x,u(x))$ for $x\in E$, if
$E_{x}\int_{0}^{\zeta}|f_u(X_{t})|\,dt<\infty$ and
\begin{equation}
\label{eq3.2} u(x)=E_{x}\Big(\int_{0}^{\zeta}f_u(X_{t})\,dt
+\int_{0}^{\zeta}dA^{\mu}_{t}\Big)
\end{equation}
for q.e. $x\in E$.
\end{definition}

In what follows we say that a function $u:E\rightarrow\BR$ is of
class (FD) if the process $t\mapsto u(X_t)$ is of class (D) under
the measure $P_x$ for q.e. $x\in E$. Similarly, we write
$u\in\FD^q$ if the process $t\mapsto u(X_t)$ belongs to $\DD^q$
under $P_x$ for q.e. $x\in E$. The notation  BSDE$_x$ means that
the backward stochastic differential equation under consideration
is defined on the filtered probability space
$(\Omega,(\FF_t)_{t\ge0},P_x)$.

\begin{theorem}
\label{th3.2} Assume that \mbox{\rm(A1), (A2), (A3${}^*$),
(A4${}^*$)} are satisfied. Then there exists a unique
probabilistic solution $u$ of \mbox{\rm(\ref{eq3.1})}. Actually,
$u$ is of class \mbox{\rm(FD)} and $u\in\FD^{q}$ for $q\in (0,1)$.
Moreover, for q.e. $x\in E$ there exists a unique solution
$(Y^{x},M^{x})$ of BSDE${}^{\zeta}_x(f+dA^{\mu})$. In fact,
\[
u(X_{t})=Y^{x}_{t},\quad t\ge 0, \quad P_{x}\mbox{-a.s.}
\]
\end{theorem}
\begin{proof}
From Lemma \ref{lem2.1} it follows that under (A4${}^*$) the
assumptions  of  Theorem \ref{th3.1} are satisfied under the
measure $P_x$ with coefficient $f(\cdot,X_\cdot)+dA^\mu$ and
terminal time $\zeta$ for q.e. $x\in E$. To prove the theorem it
now suffices to use Theorem \ref{th3.1} and repeat step by step
arguments from the proof of \cite[Theorem 4.7]{KR:JFA}.
\end{proof}
\medskip

Let us note that from Theorem \ref{th3.1} it follows that under
the assumptions of Theorem \ref{th3.2},
\begin{equation}
\label{eq3.3} E_x\int^{\zeta}_0|f_u(X_t)|\,dt\le
E_x\Big(\int^{\zeta}_0|f(X_t,0)|\,dt
+\int^{\zeta}_0d|A^{\mu}|_t\Big)
\end{equation}
for $m$-a.e. $x\in E$, where $u$ is a probabilistic solution of
(\ref{eq3.1}).

\subsection{Probabilistic solutions vs. solutions in the sense of
duality}

Assume that $(\EE,D(\EE))$ is transient and satisfies the strong
sector condition. Let $\AAA$ denote the space of all
$\EE$-quasi-continuous functions $u:E\rightarrow\BR$ such that
$u\in L^1(E;\nu)$ for every $\nu\in\hat S^{(0)}_{00}$. Following
\cite{KR:JFA} we adopt the following definition.

\begin{definition}
Let $\mu\in\MM_{0,b}$. We say  $u:E\rightarrow\BR$ is a solution
of (\ref{eq3.1}) in the sense of duality if $u\in\AAA$, $f_u\in
L^1(E;m)$ and
\begin{equation}
\label{eq3.4} \langle\nu,u\rangle=(f_u,\hat U\nu)
+\langle\mu,\widetilde{\hat U\nu}\rangle, \quad \nu\in \hat
S^{(0)}_{00}.
\end{equation}
\end{definition}

Note that by the very definition of $S^{(0)}_0$, if $\nu\in
S^{(0)}_0$ and  $u\in\FF_e$ then $\tilde u\in L^1(E;\nu)$. As a
consequence, if $u\in\FF_e$ then $\tilde u\in\AAA$.

\begin{proposition}
\label{prop3.2} Assume that $(\EE,D(\EE))$ is transient, satisfies
the strong sector condition and that $\mu\in\MM_{0,b}$. If $u$ is
$\EE$-quasi-continuous and $f_u\in L^1(E;m)$, then $u$ is a
probabilistic solution of \mbox{\rm(\ref{eq3.1})} iff it is  a
solution of \mbox{\rm(\ref{eq3.1})} in the sense of duality.
\end{proposition}
\begin{proof}
Let $u$ be a solution of (\ref{eq3.1}) in the sense of duality.
Let us denote by $w(x)$ the right-hand side of (\ref{eq3.2}) if it
is finite and set $w(x)=0$ otherwise. By Proposition
\ref{prop2.2}, $w$ is finite $m$-a.e., and hence, by Lemma
\ref{lem2.1}, $w$ is quasi-continuous. By Lemma \ref{lem2.7},
$w\in\mathcal{A}$ and $\langle\nu,w\rangle$ is equal to the
right-hand side of (\ref{eq3.4}). Thus
$\langle\nu,u\rangle=\langle\nu,w\rangle$ for $\nu\in\hat
S^{(0)}_{00}$. Lemma \ref{lem2.10} applied to the form $\hat\EE$
now shows that $u=w$ $\EE$-q.e. since $u,v$ are
$\EE$-quasi-continuous. Conversely, assume that $u$ is a
probabilistic solution of (\ref{eq3.1}). Then, again by Lemma
\ref{lem2.7}, $u\in\mathcal{A}$ and $u$ satisfies (\ref{eq3.4}).
\end{proof}

\begin{proposition}
\label{prop3.5} Assume that $(\mathcal{E}, D(\mathcal{E}))$ is
transient and \mbox{\rm(A4)} is satisfied.
\begin{enumerate}
\item[\rm(i)]
If $u$ is a probabilistic solution of \mbox{\rm(\ref{eq3.1})} then
$f_{u}\in L^{1}(E;m)$ and
\end{enumerate}
\begin{equation}
\label{eq3.5} \|f_{u}\|_{L^{1}(E;m)}
\le\|f(\cdot,0)\|_{L^{1}(E;m)}+\|\mu\|_{TV}.
\end{equation}
\begin{enumerate}
\item[\rm(ii)]If moreover $(\mathcal{E}, D(\mathcal{E}))$
satisfies the strong sector condition then $u$ is a probabilistic
solution of \mbox{\rm(\ref{eq3.1})} iff it is a solution of
\mbox{\rm(\ref{eq3.1})} in the sense of duality.
\end{enumerate}
\end{proposition}
\begin{proof}
Assertion (i) follows from (\ref{eq3.3}) and Lemma \ref{lem2.8},
whereas (ii) follows from (i) and Proposition \ref{prop3.2}.
\end{proof}

\section{Regularity of probabilistic solutions} \label{sec4}

Below, $T_k$ denotes the truncation operator defined by
(\ref{eq3.15}).

\begin{lemma}
\label{lem4.2} Assume that $(\EE,D(\EE))$ is a Dirichlet form.
Then for every $k>0$,
\begin{equation}
\label{eq4.4} \EE(T_ku,T_ku)\le \EE(u,T_ku)
\end{equation}
for all $u\in D(\EE)$. If moreover $(\EE,D(\EE))$ satisfies the
strong sector condition then \mbox{\rm(\ref{eq4.4})} holds for all
$u\in\FF_e$.
\end{lemma}
\begin{proof}
Let $u\in D(\EE)$. Since $G_\alpha$ is Markov,
\[
\alpha(T_k(u)-\alpha G_\alpha T_k(u),u-T_k(u))\ge 0,
\]
for all $k,\alpha>0$. Therefore the first assertion of the lemma
follows from \cite[Theorem I.2.13]{MR}. Now assume that $\EE$
satisfies (\ref{eq2.1}) and $u\in\FF_e$. Let $\{u_n\}\subset
D(\EE)$ be an approximating sequence for $u$. By \cite[Theorem
1.5.3]{FOT}, $T_ku_n\in\FF_e$ and
$\tilde\EE(T_ku_n,T_ku_n)\le\tilde\EE(u_n,u_n)$ for each
$n\in\BN$. Since $\{u_n\}$ is $\tilde\EE$-convergent,
$\sup_{n\ge1}\tilde\EE(T_ku_n,T_ku_n)<\infty$. Since
$(\FF_e,\tilde\EE)$ is a Hilbert space, applying the Banach-Saks
theorem we can find a subsequence $\{n_l\}$ such that the Ces\`aro
mean sequence $\{w_N=(1/N)\sum^N_{l=1}T_k(u_{n_l})\}$ is
$\tilde\EE$-convergent to  some $w\in\FF_e$. Since $\tilde\EE$ is
transient, there is an $m$-a.e. strictly positive and bounded
$g\in L^1(E;m)$ such that
\[
\int_E|w_N-v|g\,dm\le\EE(w_N-w,w_N-w)^{1/2}\rightarrow0.
\]
On the other hand, since $u_n\rightarrow u$ $m$-a.e., applying the
Lebesgue dominated convergence theorems shows that
$\int_E|w_N-T_ku|g\,dm\rightarrow0$. Consequently, $w=T_ku$ and
$\{T_ku_n\}$ converges $\tilde\EE$-weakly to $T_ku$. From this and
the first part of the proof it follows that
\begin{equation}
\label{eq4.7}
\EE(T_ku,T_ku)\le\liminf_{n\rightarrow\infty}\EE(T_ku_n,T_ku_n)
\le\liminf_{n\rightarrow\infty}\EE(u_n,T_ku_n).
\end{equation}
Moreover, using (\ref{eq2.1}) and the facts that $\{u_n\}$ is
$\tilde\EE$-convergent to $u$ and $\{T_ku_n\}$ is
$\tilde\EE$-weakly convergent to $T_ku$ we conclude the last limit
in (\ref{eq4.7}) equals $\EE(u,T_ku)$, which completes the proof
of the second assertion of the lemma.
\end{proof}

\begin{theorem}
\label{th4.1}  Assume that $(\EE,D(\EE))$ is a quasi-regular
transient  Dirichlet form  and $\mu\in\MM_{0,b}$. Then if $u$ is a
solution of \mbox{\rm(\ref{eq3.1})} and  $f_{u}\in L^{1}(E;m)$
then $T_{k}u\in\FF_e$ for every $k>0$. Moreover, for every $k>0$,
\begin{equation}
\label{eq4.2} \EE(T_{k}u,T_{k}u)\le
k(\|f_{u}\|_{L^{1}(E;m)}+\|\mu\|_{TV}).
\end{equation}
\end{theorem}
\begin{proof}
By Lemma \ref{lem2.2} there exists a nest $\{F_{n}\}$ such that
${\mathbf{1}}_{F_{n}}|f_{u}|\cdot m
+{\mathbf{1}}_{F_{n}}\cdot|\mu|\in S^{(0)}_{0}$. For $\alpha>0$
set
\[
u^{\alpha}_{n}(x)=E_{x}\int_{0}^{\zeta} e^{-\alpha t}
\fch_{F_{n}}f_{u}(X_{t})\,dt +E_{x}\int_{0}^{\zeta}e^{-\alpha t}
\fch_{F_{n}}(X_{t})\, dA_{t}^{\mu},\quad x\in E
\]
and $\mu_n={\mathbf{1}}_{F_{n}}f_{u}\cdot
m+{\mathbf{1}}_{F_{n}}\cdot\mu$. By \cite[Theorem A.8]{MMS},
\[
u^{\alpha}_n(x)=\widetilde{U_{\alpha}\mu_n}(x).
\]
for q.e. $x\in E$. Hence $u^{\alpha}_n\in D(\EE)$ and
$T_ku^{\alpha}_n\in D(\EE)$ since every normal contraction
operates on $(\tilde\EE,D(\EE))$. Therefore,
\[
\EE_{\alpha}(u^{\alpha}_n,T_ku^{\alpha}_n)
=\EE_{\alpha}(U_{\alpha}\mu_n,T_ku^{\alpha}_n)
=\int_E\widetilde{T_ku^{\alpha}_n}\,d\mu_n\le
k(\|f_u\|_{L^1(E;m)}+\|\mu\|_{TV}).
\]
By Lemma \ref{lem4.2} applied to the form $\EE_{\alpha}$,
\begin{equation}
\label{eq4.15}
\EE_{\alpha}(T_{k}u^{\alpha}_{n},T_{k}u^{\alpha}_{n})\le
\EE_{\alpha}(u^{\alpha}_n,T_{k}u^{\alpha}_{n}).
\end{equation}
Consequently,
\[
\EE(T_{k}u^{\alpha}_{n},T_{k}u^{\alpha}_{n})\le
k(\|f_u\|_{L^1(E;m)}+\|\mu\|_{TV}).
\]
By the Banach-Saks theorem we can choose a sequence $\{\alpha_l\}$
such that $\alpha_l\downarrow 0$ as $l\rightarrow\infty$, and the
sequence $\{w_N=(1/N)\sum^N_{l=1}T_ku^{\alpha_l}_n\}$ is
$\tilde\EE$-convergent. Moreover, from Lemma \ref{lem4.1} one can
deduce that $u^{\alpha}_n(x)\rightarrow u_n(x)$ as
$\alpha\downarrow0$ for q.e. $x\in E$. Hence
$T_ku^{\alpha}_n\rightarrow T_ku_n$ $m$-a.e. and consequently,
$w_N\rightarrow T_ku_n$ $m$-a.e. Thus $\{w_N\}$ is an
approximating sequence for $T_ku_n$. By (\ref{eq4.15}),
$\EE(w_N,w_N)\le k(\|f_u\|_{L^1(E;m)}+\|\mu\|_{TV})$ for every
$N\in\BN$. Hence,
\[
\EE(T_ku_n,T_ku_n)=\lim_{N\rightarrow\infty}\EE(w_N,w_N)\le
k(\|f_u\|_{L^1(E;m)}+\|\mu\|_{TV}).
\]
Since $u_n\rightarrow u$ q.e. we now apply the above arguments
again, with $T_ku^{\alpha}_n$ replaced by $T_ku_n$, to obtain
(\ref{eq4.2}).
\end{proof}

\begin{corollary}
\label{cor4.3} If $(\EE,D(\EE))$ is a quasi-regular transient
Dirichlet form and $f,\mu$ satisfy \mbox{\rm (A1), (A2), (A3*),
(A4)} then there exists a unique solution  $u$ of
\mbox{\rm(\ref{eq3.1})}. Moreover, $u$ is of class \mbox{\rm(FD)},
$u\in\FD^{q}$ for $q\in (0,1)$ and \mbox{\rm(\ref{eq3.5}),
(\ref{eq4.2})} are satisfied.
\end{corollary}
\begin{proof}
Follows immediately from Theorem \ref{th3.2}, Proposition
\ref{prop3.5} and Theorem \ref{th4.1}.
\end{proof}

\begin{remark}
\label{prop4.4} Assume that $(\EE,D(\EE))$ is transient, satisfies
the strong sector condition, and that $\mu\in S^{(0)}_0$. If $u$
is a solution of \mbox{\rm(\ref{eq3.1})} such that $f_u\cdot m\in
S^{(0)}_0$ then $u$ is a weak solution of \mbox{\rm(\ref{eq3.1})},
i.e. for every $v\in\FF_e$,
\begin{equation}
\label{eq4.8} \EE(u,v)=(f_u,v)+\langle\mu,\tilde v\rangle.
\end{equation}
Indeed, by Lemma \ref{lem2.2}, if $\mu,f_u\cdot m\in S^{(0)}_0$
then $u$ satisfying (\ref{eq3.2}) is  a quasi-continuous version
of $U(f_u\cdot m+\mu)$, which implies (\ref{eq4.8}). Note that the
condition  $f_u\cdot m\in S^{(0)}_0$ is satisfied if $f_u\in
L^2(E;m)$ and there is  $c>0$ such that $(u,u)\le c\EE(u,u)$ for
$u\in D(\EE)$. Indeed, the last inequality implies that
$S_0=S^{(0)}_0$ and from the fact that $f_u\in L^2(E;m)$ it
follows that $f_u\cdot m\in S_0$.
\end{remark}

\begin{remark}
\label{rem4.5} (i) Example 5.7 in \cite{KR:JFA} shows that in
general under (A1)--(A4) the solution $u$ of (\ref{eq3.1}) may not
be locally integrable.
\smallskip\\
(ii) Assume (A1), (A2), (A3${}^*$) (A4${}^*)$ and let $u$ be a
probabilistic solution of  (\ref{eq3.1}) as in Theorem
\ref{th3.2}. Then from (\ref{eq3.2}), (\ref{eq3.3}) it follows
that $|u(x)|\le R(|f(\cdot,0)|\cdot m+2|\mu|)$. Therefore the
condition
\begin{equation}
\label{eq4.9} (|f(\cdot,0)|,\hat G1)+\langle|\mu|,\widetilde{\hat
G1}\rangle<\infty
\end{equation}
is sufficient to guarantee integrability of $u$. One interesting
situation in which (\ref{eq4.9}) holds true is given at the end of
Section \ref{sec6}.
\end{remark}

\section{The case of semi-Dirichlet forms}
\label{sec5}

In the present section,  $E$ is a locally compact separable metric
space, $m$ is an everywhere dense positive Radon measure on
$\BB(E)$, and  $(\EE,D(\EE))$ is a transient lower-bounded
semi-Dirichlet form on $L^2(E;m)$ in the sense of \cite[Section
1.1]{O}. We also assume that $(\EE,D(\EE))$ is regular (see
\cite[Section 1.2]{O}). By $\mathbb{X}$ we denote a Hunt process
associated with $(\EE,D(\EE))$ (see \cite[Theorem 3.3.4]{O}). We
fix $\gamma>\alpha_0$, where $\alpha_0$ is the constant from
conditions ($\EE.1)$, ($\EE.2$) in \cite[Section 1.1]{O}, and we
set Cap=Cap$^{(\gamma)}$, where  Cap$^{(\gamma)}$ is the capacity
defined in \cite[Section 2.1]{O}. For $B\subset E$ we define
$\sigma_B=\inf\{t>0:X_t\in B\}$, and for $\psi\in L^1(E;m)$ we set
$P_{\psi\cdot m}(\cdot)=\int_EP_x(\cdot)\psi(x)\,m(dx)$.

\begin{lemma}
\label{lm.A1} Let $\psi\in L^1(E;m)$ be strictly positive and let
$\{A_n\}$ be a decreasing  sequence of subsets of $E$ such that
\mbox{\rm Cap}$(A_1)<\infty$. If $P_{\psi\cdot
m}(\sigma_{A_n}<\infty)\rightarrow 0$ as $n\rightarrow\infty$ then
\mbox{\rm Cap}$(A_n)\rightarrow0$.
\end{lemma}
\begin{proof}
For a Borel subset $B$ of $E$,  denote by $e^{\gamma}_B$ (resp.
$\hat e^{\gamma}_B$) the $\gamma$-equilibrium potential (resp.
$\gamma$-coequilibrium potential) of $B$ (see \cite[Section
2.1]{O} for the definitions). By \cite[Theorem 2.2.7]{O},
\[
\mbox{Cap}(A_n)=\EE_\gamma(e^\gamma_{A_n},\hat{e}^\gamma_{A_n})
\le\EE_\gamma(e^\gamma_{A_n},\hat{e}^\gamma_{A_1})
=\int_Ee^\gamma_{A_n}\,d\hat{\mu}^\gamma_{A_1},
\]
where we have used \cite[Lemma 2.1.1]{O} and the fact that
$e^\gamma_{A_n}$  is excessive. By the assumption,
$H^\gamma_{A_n}(x):= E_xe^{-\gamma\sigma_{A_n}}\searrow0$ $m$-a.e.
Hence $e^\gamma_{A_n}\searrow 0$ $m$-a.e. by \cite[Theorem
3.4.8]{O}, and consequently $e^\gamma_{A_n}\searrow0$, q.e. by
\cite[Lemma 3.4.6]{O}. This  combined with the fact that
$\hat{\mu}^\gamma_{A_1}$ is smooth gives the desired result.
\end{proof}

\begin{lemma}
\label{lm.A2} Let $A$ be a PCAF of $\mathbb{X}$ such that
$E_xA_\zeta<\infty$ for $m$-a.e. $x\in E$.  Then the assertion of
Lemma \ref{lem2.1} holds true.
\end{lemma}
\begin{proof}
By a standard argument, $u$ is finite q.e. Let $\mu$ be a smooth
measure  such that $A=A^\mu$ and let $\{F_n\}$ be a nest such that
$\mathbf{1}_{F_n}\cdot\mu\in S_0$ (see \cite[Lemma 4.1.14]{O}). By
\cite[Lemma 4.1.11]{O}, for all $\alpha>0,n\ge 1$ the function
$u^{n,\alpha}$ defined as
\[
u^{n,\alpha}(x)=E_x\int_0^\zeta
e^{-\alpha r}\mathbf{1}_{F_n}(X_r)\,dA^\mu_r
\]
is quasi continuous. Let $\psi\in L^1(E;m)$ be a strictly positive
function such that $\|\psi\|_{L^1}=1$. Then by \cite[Lemma
6.1]{BDHPS}, for $q\in(0,1)$,
\[
E_{\psi\cdot m}\sup_{t\ge 0}
|u^{n,\alpha}(X_t)-u(X_t)|^q\le \frac{1}{1-q}(E_{\psi\cdot m}
\int_0^\zeta (1-e^{-\alpha r}\mathbf{1}_{F_n}(X_r))\,dA^\mu_r)^q,
\]
Set $A^\varepsilon_{n,\alpha} =\{x\in E;
|u^{n,\alpha}(x)-u(x)|>\varepsilon\}$. Then the above inequality
yields
\[
P_{\psi\cdot m}(\sigma_{A^\varepsilon_{n,\alpha}}<\infty)
=P_{\psi\cdot m}(\sup_{t\ge
0}|u^{n,\alpha}(X_t)-u(X_t)|>\varepsilon)\rightarrow0,\quad
\alpha\rightarrow0^+,n\rightarrow\infty.
\]
By Lemma \ref{lm.A1} we have that
Cap$(A^\varepsilon_{n,\alpha})\rightarrow 0$. Because of
arbitrariness of $\varepsilon>0$, $u^{n,\alpha}\rightarrow u$
quasi-uniformly, which implies that $u$ is quasi-continuous.
\end{proof}

\begin{lemma}
For every $\phi\in L^1(E;m)\cap \BB^+(E)$ and $\mu\in S$,
\[
(R\mu,\phi)=\langle \mu,\widetilde{\hat{G}\phi}\rangle.
\]
\end{lemma}
\begin{proof}
Follows from \cite[Lemma 4.1.5]{O}.
\end{proof}

\begin{theorem}
\label{tw.A1} Assume that $\mbox{\rm{(A1), (A2), (A3*), (A4*)}}$
are satisfied. Then all the assertions of Theorem \ref{th3.2} hold
true.
\end{theorem}
\begin{proof}
It is enough to repeat the proof of Theorem \ref{th3.2} with the
only exception that we now use Lemma \ref{lm.A2} instead of Lemma
\ref{lem2.1}.
\end{proof}

We close this section with some remarks on the class $\RR$.
Proposition \ref{prop2.2} says that $\MM_{0,b}\subset\RR$  in case
$\EE$ is a transient Dirichlet form.  We shall show that for
semi-Dirichlet forms the same inclusion holds under the following
duality condition considered in \cite{K:JFA}:
\begin{enumerate}
\item[$(\Delta)$]  there exists a nest $\{F_n\}$  such that for  every
$n\in\BN$  there is a non-negative $\eta_n \in L^2(E;m)$ such that
$\eta_n>0$ $m$-a.e. on $F_n$ and $\hat{G}\eta_n$ is bounded.
\end{enumerate}

One can observe that under ($\Delta$),
\[
\RR=\{\mu\in S;R\mu<\infty\,\, m\mbox{-a.e.}\}.
\]

\begin{proposition}
\label{lem5.5} Assume that $\EE$ satisfies the duality condition
$(\Delta)$. Then $\MM_{0,b}\subset \RR$.
\end{proposition}
\begin{proof}
Let $\mu\in\MM_{0,b}$ and let $\eta_n$ be the functions of the
definition of  $(\Delta)$. Then
\[
(R\mu,\eta_n)=\langle\mu,\widetilde{\hat{G}\eta_n}\rangle \le
\|\mu\|_{TV}\cdot\|\hat{G}\eta_n\|_{\infty}<\infty
\]
for every $n\in\BN$. Hence $R\mu<\infty$ $m$-a.e., and
consequently $\mu\in\RR$ by the remark preceding the lemma.
\end{proof}

\begin{remark}
\label{rem5.6} If $\EE$ satisfies $(\Delta)$  then by Lemma
\ref{lm.A2} and Proposition \ref{lem5.5}, (A3) implies (A3${}^*$)
and (A4) implies (A4${}^*$).
\end{remark}

\section{Applications} \label{sec6}

In this section we show by examples how our general results work
in practice. Propositions \ref{prop5.5}--\ref{prop5.4} below
concerning nonlocal operators and operators in Hilbert spaces are
new even in the linear case, i.e. if $f\equiv0$. To our knowledge
Proposition \ref{prop5.2} concerning nonsymmetric local form is
also new. Note that in all the examples below concerning Dirichlet
forms one can describe explicitly the structure  of bounded smooth
measures (see \cite{KR:JMAA}). In the last two examples we
consider semi-Dirichlet forms.

\subsection{Classical nonsymmetric local regular forms}
\label{sec6.1}

We start with nonsymmetric forms  associated with divergence form
operators. Let $D$ be an bounded open subset of $\BR^d$, $d\ge3$,
and let $m$ be the Lebesgue measure on $D$. Assume that
$a:D\rightarrow\BR^d\otimes\BR^d$, $b,d:D\rightarrow\BR^d$ and
$c:D\rightarrow\BR$ are measurable functions such that
\begin{enumerate}
\item[(a)] $c-\sum^d_{i=1}\frac{\partial b_i}{\partial x_i}\ge0$ and $
c-\sum^d_{i=1}\frac{\partial d_i}{\partial x_i}\ge0$ in the sense
of Schwartz distributions,
\item[(b)]there exist $\lambda>0$, $M>0$ such that $\sum^d_{i,j=1}\tilde
a_{ij}\xi_i\xi_j\ge\lambda|\xi|^2$ for all $\xi\in\BR^d$ and
$|\check{a}_{ij}|\le M$ for $i,j=1,\dots,d$, where $\tilde
a_{ij}=\frac12(a_{ij}+a_{ji})$,
$\check{a}_{ij}=\frac12(a_{ij}-a_{ji})$,
\item[(c)]$c\in L^{d/2}_{loc}(D;dx)$ and $b_i,d_i\in
L^d_{loc}(D;dx)$, $b_i-d_i\in L^d(D;dx)\cup L^{\infty}(D;dx)$ for
$i=1,\dots,d$.
\end{enumerate}
Then by \cite[Proposition II.2.11]{MR}, the form
$(\EE,C^{\infty}_0(D))$, where
\begin{equation}
\label{eq6.9} \EE(u,v)=\int_D\langle a\nabla u,\nabla
v\rangle_{\BR^d}\,\,dx +\int_D(\langle b,\nabla
u\rangle_{\BR^d}\,v+\langle d,\nabla v\rangle_{\BR^d}\,u)\,dx
+\int_Dcuv\,dx,
\end{equation}
is closable and its closure $(\EE,D(\EE))$ is a regular Dirichlet
form on $L^2(D;dx)$. By (a) and (b),
\begin{equation}
\label{eq5.5} \EE(u,u)\ge\int_D\langle a\nabla u,\nabla
u\rangle_{\BR^d}\,dx =\int_D\langle\tilde a\nabla u,\nabla
u\rangle_{\BR^d}\,dx \ge\lambda\int_D\langle\nabla u,\nabla
u\rangle_{\BR^d}\,dx
\end{equation}
for $u\in H^1_0(D)$, and hence, by Poincar\'e's inequality, there
is $C_1>0$ such that
\begin{equation}
\label{eq5.06} \EE(u,u)\ge C_1(u,u)
\end{equation}
for $u\in H^1_0(D)$. Consequently, $(\EE,D(\EE))$ satisfies the
strong sector condition. From the calculations in \cite[pp.
50--51]{MR} it follows that there exists $C_2>0$ depending on
$\lambda$ and the coefficients $a,b,c,d$ such that
\begin{equation}
\label{eq5.6} \EE(u,u)\le C_2\BD_1(u,u),
\end{equation}
where $\BD_1(u,u)=\BD(u,u)+\int_Du^2\,dx$ and $\BD$ is defined by
(\ref{eq3.03}).  By (\ref{eq5.5})--(\ref{eq5.6}),
$D(\EE)=H^1_0(D)$. From this, (\ref{eq5.5}) and the fact that
$(\BD,H^1_0(D))$ is transient it follows that $(\EE,D(\EE))$ is
transient as well.

The operator corresponding to $(\EE,D(\EE))$ in the sense of
(\ref{eq1.2}) has the form
\begin{equation}
\label{eq6.11} Lu=\sum^d_{i,j=1}\frac{\partial}{\partial
x_i}(a_{ij}\frac{\partial u}{\partial x_j})
-\sum^d_{i=1}b_i\frac{\partial u}{\partial x_i}
+\sum^d_{i=1}\frac{\partial}{\partial x_i}(d_iu)-cu.
\end{equation}

From the above considerations and Corollary \ref{cor4.3} we obtain
the following proposition.
\begin{proposition}
\label{prop5.2} Let $D\subset\BR^d$, $d\ge3$, be a bounded domain
and let $a,b,c,d$ satisfy \mbox{\rm(a)--(c)}. If $f,\mu$ satisfy
\mbox{\rm(A1)--(A4)} then there exists a unique probabilistic
solution of the problem
\[
-Lu=f_u+\mu \mbox{ on }D,\quad u|_{\partial D}=0.
\]
Moreover, $f_u\in L^1(D;dx)$, $T_ku\in H^1_0(D)$ for every $k>0$
and \mbox{\rm(\ref{eq3.5}), (\ref{eq4.2})} hold true.
\end{proposition}

\subsection{Gradient perturbations of nonlocal symmetric regular
forms on $\BR^d$} \label{sec6.2}

The following example of a  nonlocal nonsymmetric regular
Dirichlet form is borrowed from \cite{J1}.

Let $\psi:\BR^d\rightarrow\BR$ be a continuous negative definite
function, i.e $\psi(0)\ge0$ and $\xi\mapsto e^{-t\psi(\xi)}$ is
positive definite for $t\ge0$, and for $s\in\BR$ let $H^{\psi,s}$
denote the Hilbert space
\[
H^{\psi,s}=\{u\in L^2(\BR^d;dx):\|u\|_{\psi,s}<\infty\},
\]
where
\[
\|u\|^2_{\psi,s}= \int_{\BR^d}(1+\psi(\xi))^{s}|\hat
u(\xi)|^2\,d\xi
\]
and $\hat u(\xi)=(2\pi)^{-d/2}\int_{\BR^d}e^{-i\xi\cdot
x}u(x)\,dx$, $\xi\in\BR^d$. If $\psi(\xi)=|\xi|^2$ then
$H^{\psi,s}$ coincides with the usual fractional Sobolev space
$H^{s}$. Basic properties of the spaces $H^{\psi,s}$ are found in
\cite[Section 3.10]{J1}.

Given $\psi$ as above and
$b=(b_1,\dots,b_d):\BR^d\rightarrow\BR^d$ such that $b_i\in
C^1_b(\BR^d)$ for $i=1,\dots,d$  define forms $\Psi,\BB$ by
\[
\Psi(u,v)=\int_{\BR^d}\psi(\xi)\hat u(\xi)\overline{\hat
v(\xi)}\,d\xi,\quad u,v\in H^{\psi,1},
\]
\begin{equation}
\label{eq5.12} \BB(u,v)=\int_{\BR^d}\langle b,\nabla
u\rangle_{\BR^d}\,v\,dx,\quad u,v\in C_0^{\infty}(\BR^d).
\end{equation}
Consider the following assumptions on $\psi,b$:
\begin{enumerate}
\item[(a)] $1/\psi$ is locally integrable on $\BR^d$,
\item[(b)] Tthere exist $\alpha\in(1,2]$, $\kappa>0$, $R>0$ such that
$\psi(\xi)\ge\kappa|\xi|^{\alpha}$ if $|\xi|>R$,
\item[(c)]$b_i\in C^1_b(\BR^d)$ for $i=1,\dots,d$ and $\dyw\,b=0$.
\end{enumerate}
It is known (see, e.g., \cite[Example 1.4.1]{FOT}) that
$(\Psi,H^{\psi,1})$ is a symmetric regular Dirichlet form on
$L^2(\BR^d;dx)$. By \cite[Example 1.5.2]{FOT} it is transient iff
condition (a) is satisfied. By \cite[Corollary 4.7.35]{J1} there
exists $c>0$ (depending on $b$) such that $|\BB(u,v)|\le
c\|u\|_{H^{1/2}}\|v\|_{H^{1/2}}$. Hence, if (b) is satisfied then
$H^{\psi,1}\subset H^{1/2}$ and
\begin{equation}
\label{eq5.1} |\BB(u,v)|\le C\|u\|_{\psi,1}\|v\|_{\psi,1}
\end{equation}
for some $C>0$. Since $C^{\infty}_0(\BR^d)$ is dense in
$H^{\psi,1}$ (see \cite[Theorem 3.10.3]{J1}), it follows that
under (b) we may extend (\ref{eq5.12}) to a  continuous bilinear
form $(\BB,H^{\psi,1})$. If moreover (c) is satisfied, then by
integration by parts,
\begin{equation}
\label{eq5.2} \BB(u,u)=-\frac12\int_{\BR^d}u^2\dyw\,b\,dx=0
\end{equation}
for $u\in C^{\infty}_0(\BR^d)$ and hence for all $u\in
H^{\psi,1}$. Using  integration by parts one can also check (see
\cite[Example 4.7.36]{J1}) that if $\dyw\,b=0$ then
$(\BB,H^{\psi,1})$ has the contraction properties required in the
definition of a Dirichlet form and hence is a Dirichlet form.
Finally, let us consider the form
\begin{equation}
\label{eq5.10} \EE(u,v)=\Psi(u,v)+\BB(u,v),\quad u,v\in
H^{\psi,1}.
\end{equation}
From the properties of $\Psi,\BB$ mentioned above it follows that
if (a)--(c) are satisfied then $(\EE,H^{\psi,1})$ is a regular
transient Dirichlet form on $L^2(\BR^d;dx)$ and the extended
Dirichlet space for $\EE$ coincides with the extended Dirichlet
space  for $\Psi$, which we denote here by $H^{\psi,1}_e$. The
space $H^{\psi,1}_e$ can be characterized for $\psi$ of the form
$\psi(\xi)=c|\xi|^{\alpha}$ for some $\alpha\in(0,2]$, $c>0$ (see
\cite[Example 1.5.2]{FOT} or \cite[Example 3.5.55]{J2}). That
characterization shows that if $\psi$ satisfies (b) and $\alpha<d$
(i.e. (a) is satisfied) then $H^{\psi,1}_e\hookrightarrow
L^p(\BR^d)$ with $p=2d/(d-\alpha)$ and $\|u\|_{L^p(\BR^d;dx)}\le
C\Psi(u,u)^{1/2}$ for $u\in H^{\psi,1}_e$ (see \cite[Corollary
3.5.60]{J2}).

The operator associated with $\Psi$ is a pseudodifferential
operator $\psi(\nabla)$ which for $u\in C^{\infty}_0(\BR^d)$ has
the form
\[
\psi(\nabla)u(x)=(2\pi)^{-d/2}\int_{\BR^d}e^{i(x,\xi)}\psi(\xi)\hat
u(\xi)\,d\xi,\quad x\in \BR^d.
\]
\begin{proposition}
\label{prop5.5} Assume that \mbox{\rm(A1)--(A4)} and
\mbox{\rm(a)--(c)} hold. Then there exists a unique probabilistic
solution of the equation
\[
-\psi(\nabla)u-(b,\nabla u)=f_u+\mu.
\]
Moreover, $f_u\in L^1(\BR^d;dx)$, $T_ku\in H^{\psi,1}_e$ for every
$k>0$, and \mbox{\rm(\ref{eq3.5}), (\ref{eq4.2})} are satisfied.
\end{proposition}

Proposition \ref{prop5.5}  holds for operators corresponding to
(\ref{eq5.10}) with $\Psi$ replaced by an arbitrary symmetric
regular Dirichlet form with domain $H^{\psi,1}$. For examples of
such forms  see Examples 4.7.30 and 4.7.31 in \cite{J1} and Remark
2.6.8 and Theorem 2.6.10 in \cite{J2}.

\subsection{Nonlocal symmetric forms on $D\subset\BR^d$}
\label{sec6.4}

Let $\psi$ be a continuous negative definite function satisfying
conditions (a), (b) of Subsection \ref{sec6.2}, and let
$D\subset\BR^d$ be a nearly Borel measurable set finely open with
respect to the process associated with the form $\Psi$. Set
$L^2_D(\BR^d;dx)=\{u\in L^2(\BR^d;dx):u=0$ a.e. on $D^c$\} and
\[
H^{\psi,1}_D=\{u\in H^{\psi,1}:\tilde u=0\mbox{ q.e. on }D^c\}.
\]
Then by  \cite[Theorem 3.3.8]{CF}, $(\Psi,H^{\psi,1}_D)$ is a
quasi-regular Dirichlet form on $L^2_D(\BR^d;dx)$. If $\alpha<d$
then it is transient by \cite[Theorem 4.4.4]{FOT}. In case $\Psi$
is transient, we denote its extended Dirichlet space by
$H^{\psi,1}_{D,e}$. The above remarks and Corollary \ref{cor4.3}
lead to the following proposition.

\begin{proposition}
\label{prop5.6} Let assumptions of Proposition \ref{prop5.5} hold
and let $D$ be a nearly Borel finely open subset of $\BR^d$ with
$d>\alpha$. Then there exists a unique probabilistic solution of
the problem
\begin{equation}
\label{eq5.11} -\psi(\nabla)u=f_u+\mu\quad\mbox{in }D,\quad
u=0\quad\mbox{in }D^c.
\end{equation}
Moreover, $f_u\in L^1(D;dx)$, $T_ku\in H^{\psi,1}_{D,e}$ for every
$k>0$ and \mbox{\rm(\ref{eq3.5}), (\ref{eq4.2})} hold true.
\end{proposition}

Let us remark that if $D$ is bounded then
$H^{\psi,1}_{D,e}=H^{\psi,1}_D$, because
$H^{\psi,1}_D\hookrightarrow L^2_D(\BR^d;dx)$ in that case. If $D$
is open and has smooth boundary then as in \cite{JM} we may define
the space $H^{\psi,1}_0(D)$ as follows. Given $u\in
C^{\infty}_0(D)$ we extend it to $\BR^d$ by setting $u=0$ on
$D^c$. We then obtain a function $u\in C^{\infty}_0(\BR^d)$ with
support in $D$. Consequently, we may regard $C^{\infty}_0(D)$ as a
subspace of $H^{\psi,1}$ and therefore define $H^{\psi,1}_0(D)$ as
the closure of $C^{\infty}_0(D)$ in $H^{\psi,1}$. By \cite[Theorem
4.4.3]{FOT}, $C^{\infty}_0(D)$ is a special standard core of
$(\Psi,H^{\psi,1}_D)$, and hence, by \cite[Lemma 2.3.4]{FOT},
$H^{\psi,1}_D=H^{\psi,1}_0(D)$.

Assume that $d\ge3$ and $D\subset\BR^d$ is a bounded open set with
a $C^{1,1}$ boundary. Let us consider the form
$(\Psi,H^{\psi,1}_D)$ with $\psi(\xi)=c|\xi|^{\alpha}$ for some
$c>0, \alpha\in(0,2]$. By \cite[Proposition 4.9]{Ku} there exist
constants $0<c_1<c_2$ depending only on $d,\alpha,D$ such that
\[
c_1\delta^{\alpha/2}(x)\le R1(x)\le c_2\delta^{\alpha/2}(x),\quad
x\in D,
\]
where $\delta(x)=\mbox{dist}(x,\partial D)$. From this, Theorem
\ref{th3.2} and Remark \ref{rem4.5} it follows that if $f$
satisfies (A1), (A2), (A3${}^*$) and $f(\cdot,0)\in
L^1(D;\delta^{\alpha/2}(x)\,dx)$,
$\int_D\delta^{\alpha/2}(x)|\mu|(dx)<\infty$ then the
probabilistic solution $u$ of (\ref{eq3.1}) belongs to
$L^1(D;dx)$.

\subsection{Dirichlet forms on infinite dimensional state space}
\label{sec6.3}

Let $H$ be a separable real Hilbert space and let $A,Q$ be linear
operators on $H$. Assume that
\begin{enumerate}
\item[(a)]$A:D(A)\subset H\rightarrow H$ generates a strongly
continuous semigroup $\{e^{tA}\}$ in $H$ such that $\|e^{tA}\|\le
Me^{-\omega t}$, $t\ge0$, for some $M>0$, $\omega>0$,
\item[(b)]$Q$ is bounded, $Q=Q^*>0$  and
$\sup_{t>0}\tr\,Q_t<\infty$, where
$Q_t=\int^t_0e^{sA}Qe^{sA^*}\,ds$,
\item[(c)]$Q_{\infty}(H)\subset D(A)$, where
$Q_{\infty}=\int^{\infty}_0e^{tA}Qe^{tA^*}\,dt$.
\end{enumerate}

A simple and important example of $A,Q$ satisfying (a)--(c) is
$Q=I$ and a self-adjoint operator $A$ such that $\langle
Ax,x\rangle_H\le-\omega|x|_H^2$, $x\in D(A)$, for some $\omega>0$
and $A^{-1}$ is of trace class. In this example
$Q_{\infty}=-\frac12 A^{-1}$. Other examples are  found for
instance in \cite{F}.

By (a) the operators $Q_t$, $Q_{\infty}$ are well defined, and by
(b), $Q_{\infty}$ is of trace class. Let $\gamma$ denote the
Gaussian measure on $H$ with mean 0 and covariance operator
$Q_{\infty}$. We consider the form
\begin{equation}
\label{eq5.3} \EE(u,v)=-\int_H\langle\nabla u,AQ_{\infty}\nabla
v\rangle_H\,d\gamma, \quad u,v\in\FF C^{\infty}_b.
\end{equation}
Here $\FF C^{\infty}_b$ is the space of finitely based smooth
bounded functions, i.e.
\[
\FF C^{\infty}_b=\{u:H\rightarrow\BR: u(x)=f(\langle x,e_1\rangle,
\dots,\langle x,e_m\rangle),m\in\BN,f\in C^{\infty}_b(\BR^m)\}
\]
for some orthonormal basis $\{e_k\}$ of $H$ consisting of
eigenvectors of $Q_{\infty}$, and $\nabla$ is the $H$-gradient
defined for $u\in\FF C^{\infty}_b$ as the unique element of $H$
such that $\langle\nabla u(x),h\rangle_H=\frac{\partial
u}{\partial h}(x)$ for $x\in H$ (the last derivative is the
Gateaux derivative in the direction $h$, i.e. $\frac{\partial
u}{\partial h}(x)=\frac{d}{ds}u(x+sh)|_{s=0}$). Under (a)--(c) the
form $(\EE,\FF C^{\infty}_b)$ is closable and its closure, which
will be denoted by  $(\EE,W^{1,2}_Q(H))$, is a coercive closed
form on $L^2(H;\gamma)$ (see Theorem 2.2, Remark 2.3 and Lemma 3.3
in \cite{F}).
Using the product rule for $\nabla$ on $\FF C^{\infty}_b$ one can
check in the same way as in \cite[Section II.2(d)]{MR} (see also
\cite[Section II.3(e)]{MR}) that it has the Dirichlet property.
Finally, by results of \cite[Section IV.4]{MR}, it is
quasi-regular.

By \cite[Theorem 3.6]{F} the semigroup $\{P_t\}$ on
$L^2(H;\gamma)$ associated with $(\EE,W^{1,2}_Q(H))$ is the
Ornstein-Uhlenbeck semigroup of the form
\[
P_tf(x)=\int_Hf(y)\NN(e^{tA}x,Q_t)\,(dy), \quad x\in H,
\]
where $\NN(e^{tA}x,Q_t)$ is the gaussian measure in $H$ with mean
$e^{tA}x$ and covariance operator $Q_t$. Note that $\{P_t\}$ is
analytic. Actually, analyticity of $\{P_t\}$ is equivalent to the
fact that it corresponds to some nonsymmetric Dirichlet form (see
\cite{G} and also \cite{GN} for related results in a more general
setting). The generator of $\{P_t\}$ has the form
\[
Lu(x)=\frac12\tr({Q\Delta u(x)})+\langle x,A^{*}\nabla
u(x)\rangle_H,\quad x\in H.
\]

Since for every $\lambda>0$ the form
$(\EE_{\lambda},W^{1,2}_Q(D))$ is transient, from the above
remarks and Corollary \ref{cor4.3} we obtain the following
proposition.
\begin{proposition}
\label{prop5.4} Assume that \mbox{\rm(A1)--(A4)} and
\mbox{\rm(a)--(c)} hold. Then for every $\lambda>0$ there exists a
unique probabilistic solution to the equation
\[
-Lu+\lambda u=f_u+\mu.
\]
Moreover, $f_u\in L^1(H;\gamma)$, $T_ku\in W^{1,2}_Q(D)$ for every
$k>0$ and \mbox{\rm(\ref{eq3.5}), (\ref{eq4.2})} hold true.
\end{proposition}

For generalizations of forms (\ref{eq5.3})  to operators $Q$
depending on $x$ or more general measures on topological vector
spaces than gaussian measures on Hilbert spaces we refer the
reader to \cite[Section II.3]{MR}, \cite{G,R} and  the references
therein).

\subsection{Additional remarks on Dirichlet forms}

In this subsection we briefly outline how general results on
transformation of Dirichlet forms can by applied to obtain other
interesting examples of quasi-regular forms.
\medskip\\
(i) (Perturbation of Dirichlet forms) Let $(\EE,D(\EE))$ be a
quasi-Dirichlet form and let $\nu\in S$. Set
\[
\EE^{\nu}(u,v)=\EE(u,v)+\int_E\tilde u\tilde v\,d\nu,\quad u,v\in
D(\EE^{\nu}),
\]
where $D(\EE^{\nu})=D(\EE)\cap L^2(E;\nu)$. By \cite[Proposition
2.3]{RS}, $(\EE^{\nu},D(\EE^{\nu}))$ is a quasi-regular Dirichlet
form on $L^2(E;m)$. In our context an important example of $\nu$
is $\nu(dx)=V(x)\,m(dx)$ for some nonnegative $V\in L^1(E;m)\cap
L^{\infty}(E;m)$. In this case $D(\EE^V)\equiv
D(\EE^{\nu})=D(\EE)$. Moreover, $(\EE^V,D(\EE^V))$ satisfies the
strong sector condition if $(\EE,D(\EE))$ does, and from
(\ref{eq2.3}) it follows immediately that $(\EE^V,D(\EE^V))$ is
transient if $(\EE,D(\EE))$ is transient or $V$ is $m$-a.e.
strictly positive. Therefore Propositions \ref{prop5.2} and
\ref{prop5.4} hold true for operators $L$ replaced by $L-V$ (In
Proposition \ref{prop5.4} we can take $\lambda\ge0$ if $V$ is
$m\equiv\mu$-a.e. strictly positive), and Proposition
\ref{prop5.6} holds for $\psi(\nabla)$ replaced by
$\psi(\nabla)-V$. Note that the perturbed regular form may become
non-regular. For instance, in \cite[Section II.2(e)]{MR} one can
find an  example of $V$ such that the classical form
$(\BD,H^1(\BR^d))$ (see (\ref{eq3.03})) perturbed by $V$ is not
regular.
\medskip\\
(ii) (Superposition of closed forms) For $k=1,2$ let
$(\EE^{(k)},D^{(k)})$ be a closable symmetric bilinear form on
$L^2(E;m)$. Set
\[
\EE(u,v)=\EE^{(1)}(u,v)+\EE^{(2)}(u,v),\quad u,v\in D,
\]
where $D=\{u\in D^{(1)}\cap D^{(2)}:\EE^{(1)}(u,u)+\EE^{(2)}(u,u)
<\infty\}$. By \cite[Proposition I.3.7]{MR} the form $(\EE,D)$ is
closable on $L^2(E;m)$. We may use this property and examples
considered in Section \ref{sec6.1}--\ref{sec6.3} to construct new
quasi-regular Dirichlet forms. To illustrate how this work in
practice,  following \cite[Remark II.3.16]{MR} we consider the
form $(\EE,\FF C^{\infty}_b)$ of Section \ref{sec6.1} and a
symmetric finite positive measure on $(H\times
H,\BB(H)\otimes\BB(H))$ such that the form
\[
J(u,v)=\int_H\int_H(u(x)-u(y))(v(x)-v(y))\,J(dx\,dy),\quad u,v\in
\FF C^{\infty}_b,
\]
is closable. Then the form $(\EE+J,\FF C^{\infty}_b)$ is closable
and its closure is a symmetric quasi-regular Dirichlet form. Thus
we have constructed an infinite-dimensional (and so non-regular)
Dirichlet form which is nonlocal. For the operator corresponding
to that form one can formulate an analogue of Proposition
\ref{prop5.4}.

General results on superposition of closed form are  found in
\cite[Section 3.1]{FOT} and \cite[Proposition I.3.7]{MR}.
\medskip\\
(iii) (Parts of forms) Let $(\EE,D(\EE))$ be a symmetric regular
Dirichlet form on $L^2(E;m)$, and  $D\subset E$ be a nearly Borel
measurable finely open set with respect to the process $\BX$
associated with $(\EE,D(\EE))$. Set $L^2_D(E;m)=\{u\in
L^2(E;m):u=0$ $m$-a.e. on $D^c$\} and $\FF_D=\{u\in D(\EE):\tilde
u=0$ q.e. on $D^c$\}. By \cite[Theorem 3.3.8]{CF} the form
$(\EE,\FF_D)$ on $L^2_D(E;m)$, called the part of $(\EE,D(\EE))$
on $D$, is a quasi-regular Dirichlet form (if $D$ is open then it
is regular). We can use this result to get solutions of Dirichlet
problems of the form (\ref{eq5.11}) with $\psi(\nabla)$ replaced
by arbitrary operator  associated with a symmetric regular
Dirichlet form.

\subsection{Semi-Dirichlet forms}

\paragraph{Diffusion operator with drift.} Let $D\subset\BR^d$, $d\ge3$,
be a bounded domain and let $a_{ij}, b_i:D\rightarrow\BR$ be
measurable functions such that $b_i$ is bounded, $a_{ij}=a_{ji}$
and
\[
\lambda^{-1} |\xi|^{2} \le \sum^d_{i,j=1}a_{ij}\xi_{i}\xi_{j}\le
\lambda|\xi|^{2}, \quad \xi=(\xi_1,\dots,\xi_d)\in\BR^d,
\]
for some $\lambda\ge1$. Consider the form $(\EE,C^{\infty}_0(D))$
defined by (\ref{eq6.9}) with $c=0$, $d=0$.
By Theorems 1.5.2 and 1.5.3 in \cite{O} its smallest closed
extension $(\EE,H^1_0(D))$  is a regular lower-bounded
semi-Dirichlet form on $L^2(D;dx)$.
Therefore, if (A1), (A2), (A3${}^*$), (A4${}^*$) are satisfied,
then there exists a unique probabilistic solution of (\ref{eq1.1})
with $L$ defined by (\ref{eq6.11}) with $c=0$, $d=0$.

Let $G_D$ denote the Green function for $L$ on $D$ and let $\hat
G_D$ denote the Green function on $D$ for the adjoint operator to
$L$, i.e. operator associated with the form $(\hat\EE,H^1_0(D))$.
It is known that $G_D(x,y)=\hat G_D(y,x)$ and $G_D(x,y)\le
c|x-y|^{-(d-2)}$ for $x,y\in D$ such that $x\neq y$ (see, e.g.,
\cite[Section 4.2]{P}). Therefore,
\[
\hat G1(x)=\int_D\hat G_D(x,y)\,dy=\int_DG_D(y,x)\,dy \le c
\int_D|x-y|^{-d+2}\,dy,
\]
and hence
\[
\hat G1(x)\le c\int_{B(x,\mbox{\tiny
diam}(D))}|x-y|^{-d+2}\,dy=c_1(\mbox{diam}(D))^2.
\]
Accordingly, $\EE$ satisfies condition $(\Delta)$ with $\eta_n=1$
and $F_n=D$. From this and Remark \ref{rem5.6} it follows that
(A3) implies (A3${}^*$) and (A4) implies (A4${}^*$).

\paragraph{Fractional laplacian with variable exponent.}

Let $\alpha:\BR^d\rightarrow\BR$ be a measurable function such
that $\alpha_1\le\alpha(x)\le\alpha_2$, $x\in\BR^d$, for some
constants $0<\alpha_{1}\le\alpha_{2}<2$.  Let
$L_{t}=L=\Delta^{\alpha(x)}$, i.e. $L$ is a pseudodifferential
operator such that
\begin{equation}
\label{eq6.10} Lu(x)=(2\pi)^{-d/2}\int_{\BR^d}
e^{ix\xi}|\xi|^{\alpha(x)}\hat{u}(\xi)\,d\xi,\quad u\in
C^{\infty}_c(\BR^d).
\end{equation}
If $\int_0^1(\beta(r)|\log r|)^2r^{-(1+\alpha_2)}\,dr<\infty$,
where $\beta(r)=\sup_{|x-y|\le r}|\alpha(x)-\alpha(y)|$, then $L$
is associated with some regular semi-Dirichlet form $\EE$ on
$L^2(\BR^d;dx)$ (see \cite[Example 5.13]{K:JFA} for details).
Therefore under the above assumptions on $\alpha$ and (A1), (A2),
(A3${}^*$), (A4${}^*$) there exists a unique probabilistic
solution of (\ref{eq1.1}) with $L$ defined by (\ref{eq6.10}).
\medskip\\
{\bf Acknowledgments}
\medskip\\
The first author was supported by Polish NCN grant no.
2012/07/D/ST1/02107.


\begin{thebibliography}{aa}

\bibitem{BBGGPV}
P. B\`enilan, L. Boccardo, T. Gallou\"et, R. Gariepy, M. Pierre
and J.L. Vazquez, {\em An $L^1$-theory of existence and uniqueness
of solutions of nonlinear elliptic equations}, Ann. Scuola Norm.
Sup. Pisa Cl. Sci. (4) {\bf 22} (1995), 241--273.

\bibitem{BDHPS}
P. Briand, B. Delyon, Y. Hu, E. Pardoux and L. Stoica, {\em $L^p$
solutions of backward stochastic differential equations},
Stochastic Process. Appl. {\bf 108} (2003), 109--129.

\bibitem{CF}
Z.-Q. Chen and M. Fukushima, {\em Symmetric Markov Processes, Time
Change, and Boundary Theory}, Princeton University Press,
Princeton, 2012.

\bibitem{DPZ}
G. Da Prato and J. Zabczyk, {\em Second Order Partial Differential
Equations in Hilbert Spaces}, Cambridge University Press,
Cambridge 2002.

\bibitem{FSW}
M. Fila, Ph. Souplet and F.B. Weissler,  {\em Linear and nonlinear
heat equations in $L^q_{\delta}$  spaces and universal bounds for
global solutions}, Math. Ann. {\bf 320} (2001), 87--113.

\bibitem{F}
M. Fuhrman, {\em Analyticity of transition semigroups and
closability of bilinear forms in Hilbert spaces}, Studia Math.
{\bf 115} (1995), 53--71.

\bibitem{FOT}
M. Fukushima, Y. Oshima and M. Takeda, {\em Dirichlet Forms and
Symmetric Markov Processes}, Walter de Gruyter, Berlin, 1994.

\bibitem{G}
B. Goldys,  {\em On analyticity of Ornstein-Uhlenbeck semigroups},
Atti Accad. Naz. Lincei Cl. Sci. Fis. Mat. Natur. Rend. Lincei (9)
Mat. Appl. {\bf 10} (1999), 131--140.

\bibitem{GN}
B. Goldys and J.M.A.M. van Neerven, {\em Transition semigroups of
Banach space-valued Ornstein-Uhlenbeck processes}, Acta Appl.
Math. {\bf 76} (2003), 283--330.

\bibitem{J1}
N. Jacob, {\em Pseudo-Differential Operators and Markov Processes.
Vol. I: Fourier Analysis and Semigroups}, Imperial College Press,
London, 2001.

\bibitem{J2}
N. Jacob, {\em Pseudo-Differential Operators and Markov Processes.
Vol. II: Generators and Their Potential Theory}, Imperial College
Press, London, 2002.

\bibitem{JM}
N. Jacob and V. Moroz, {\em On the semilinear Dirichlet problem
for a class of nonlocal operators generating Dirichlet forms},
Progr. Nonlinear Differential Equations Appl. {\bf 40} (2000),
191--204.

\bibitem{K:JFA}
T. Klimsiak, {\em Semi-Dirichlet forms, Feynman-Kac functionals
and the Cauchy problem for semilinear parabolic equations}, J.
Funct. Anal. {\bf 268} (2015), 1205--1240.

\bibitem{KR:JFA}
T. Klimsiak and A. Rozkosz, {\em Dirichlet forms and semilinear
elliptic equations with measure data}, J. Funct. Anal. {\bf 265}
(2013), 890--925.

\bibitem{KR:JMAA}
T. Klimsiak and A. Rozkosz, {\em On the structure of bounded
smooth measures associated with  quasi-regular Dirichlet form},
arXiv:1410.4927.

\bibitem{Ku}
T. Kulczycki, {\em Properties of Green function of symmetric
stable processes}, Probab. Math. Statist. {\bf 17} (1997),
339--364.

\bibitem{MOR}
Z.-M. Ma, L. Overbeck and M. R\"ockner, {\em Markov processes
associated with semi-Dirichlet forms}, Osaka J. Math. {\bf 32}
(1995), 97--117.

\bibitem{MMS}
Li Ma, Zhi-Ming Ma and Wei Sun, {\em Fukushima's decomposition for
diffusions associated with semi-Dirichlet forms}, Stoch. Dyn. {\bf
12}, 1250003, (2012), 31 pp.

\bibitem{MR}
Z.-M. Ma and M. R\"ockner, {\em Introduction to the Theory of
(Non-Symmetric) Dirichlet Forms}, Springer-Verlag, Berlin, 1992.

\bibitem{O}
Y. Oshima, {\em Semi-Dirichlet Forms and Markov Processes}, Walter
de Gruyter, Berlin, 2013.

\bibitem{P}
R.G. Pinsky, {\em Positive harmonic functions and diffusion},
Cambridge University Press, Cambridge, 1995.

\bibitem{R}
M. R\"ockner, {\em $L^p$-analysis of finite and infinite
dimensional diffusion operators}, in: Stochastic PDE's and
Kolmogorov Equations in Infinite Dimensions (Cetraro, Italy 1998),
G. Da Prato (ed.), Lecture Notes in Math. {\bf 1715}, Springer,
Berlin, 1999, 65--116.

\bibitem{RS}
M. R\"ockner and B. Schmuland, {\em Quasi-regular Dirichlet forms:
examples and counterexamples}, Canad. J. Math. {\bf 47} (1995),
165--200.

\end{thebibliography}
\end{document}